\documentclass[conference]{asl}
\usepackage{cabal}
\usepackage{xypic, amssymb} 

\renewcommand{\phi}{\varphi}
\newcommand{\bfb}{{\mathbf b}}

\DeclareMathOperator{\gopen}{\Game_{\text{open-$\omega_1$}}\bfPi^1_1}
\DeclareMathOperator{\true}{\tt true} 
\DeclareMathOperator{\lRK}{\leq_{RK}}

\DeclareMathOperator{\dGuB}{\Gamma_{uB}}
\newcommand{\cU}{\mathcal U}
\newcommand{\cV}{\mathcal V}

\newcommand{\cO}{\mathcal O}

\newcommand{\cA}{\mathcal A}
\newcommand{\cT}{\mathcal T}
\newcommand{\cTdg}{\mathcal T_{\delta,\gamma}}
\newcommand{\cTdo}{\mathcal T_{\delta,\omega}}
\newcommand{\cTdd}{\mathcal T_{\delta,\delta}}
\newcommand{\CL}{\lang}
\newcommand{\CLdg}{\CL_{\delta,\gamma}}
\newcommand{\cW}{\mathcal W}
\newcommand{\cWdg}{\cW_{\delta,\gamma}}
\newcommand{\cWdo}{\cW_{\delta,\omega}}
\newcommand{\cWdd}{\cW_{\delta,\delta}}

\newcommand{\bbR}{\R}
\newcommand{\bbP}{\mathbb P}
\newcommand{\bbQ}{\mathbb Q}

\newcommand{\cB}{{\mathcal B}}   
\newcommand{\Bdg}{\cB_{\delta,\gamma}}
\newcommand{\cP}{{\mathcal P}}
\newcommand{\bfC}{{\mathbf C}}
\newcommand{\bbC}{{\mathbb C}}
\newcommand{\fM}{\mathcal  M}
\DeclareMathOperator{\mwM}{{\mathcal M}_{mw}}
\DeclareMathOperator{\Hom}{Hom}

\renewcommand{\>}{\rangle}

\newcommand{\x}{\vec x}

\renewcommand\labelenumi{(\arabic{enumi})}

\theoremstyle{definition}

\newcommand{\lthen}{\rightarrow}

\newcounter{my_enumerate_counter}
\newcommand{\pushcounter}{\setcounter{my_enumerate_counter}{\value{enumi}}}
\newcommand{\popcounter}{\setcounter{enumi}{\value{my_enumerate_counter}}}

\DeclareMathOperator{\Ult}{Ult}
\DeclareMathOperator{\Coll}{Col}

\newcommand{\Coqq}[2]{\Coll(#1,#2)}
\def\rs{\restriction}

\newcommand{\forces}{\Vdash}

\DeclareMathOperator{\Int}{int}

\newcommand{\lbl}{\label}
\renewcommand{\land}{\wedge}

\newcommand{\Wdd}{{\mathcal W}_{\delta,\delta}}

\title{The extender algebra and $\Sigma^2_1$-absoluteness} 
\author{Ilijas Farah} 

\address{Department of Mathematics and Statistics\\
York University\\
Toronto, ON, Canada\\
and\\
Matemati\v cki Institut SANU\\ Kneza Mihaila 36\\ Belgrade, Serbia}

\email{ifarah@mathstat.yorku.ca}

\urladdr{http://www.math.yorku.ca/$\sim$ifarah}
\thanks{Partially supported by NSERC}

\date{December 14, 2005. Version of \today.} 
\begin{document} 
\begin{abstract} 
We present  a self-contained account  
of Woodin's extender algebra and its use
in proving absoluteness results, including a proof of the $\Sigma^2_1$-absoluteness theorem. 
We also include a proof that the existence of an inner model with Woodin limit of Woodin cardinals
implies the existence of divergent models of $\AD^+$. 
\end{abstract} 

\maketitle
This note provides an introduction to Woodin's 
 extender algebra and a proof (due to Steel and Woodin independently)  
of Woodin's $\Sigma^2_1$-absoluteness theorem, using the extender algebra, 
 from a large cardinal assumption. 
Unlike the published accounts of this proof, 
the present account 
 should be accessible to a set theorist familiar with forcing and basic 
large cardinals (\eg, \cite{Ku:Book}, \cite{Kana:Book}). 
In particular, no familiarity with the  inner model theory is required. 
A more comprehensive  account of the extender algebra can also be found 
in \cite[\S7.2]{steel06innermodel} and 
\cite[\S4]{neeman04longgames} and the reader is invited to consult these excellent sources for 
more information and other applications. 
A strengthening of the $\Sigma^2_1$ absoluteness in terms of determinacy 
of games of length~$\omega_1$ was proved by Neeman (\cite{Nee:Games})
also using the extender algebra. In \cite{Lar:Three},  the 
extender algebra was used to prove the consistency of Woodin's $\Omega$-conjecture. 

\subsection*{Organization of the paper} 
In \S\ref{S.Logics} we introduce infinitary propositional logic $\CL_{\delta,\gamma}$. 
In \S\ref{S.EE} we review basics of elementary embeddings, Woodin cardinals and extenders. 
The extender algebra is introduced in \S\ref{S.EA}, where we also prove that is powerfully $\delta$-cc 
and introduce iteration trees and the iteration game. A variety of genericity iteration theorems is 
proved in \S\ref{S.GI} and their applications to absoluteness are proved in \S\ref{S.Ab}. 
Divergent models of $\AD^+$ are constructed in \S\ref{S.Div},
and \S\ref{S.Sigma22} has no real content.

\subsection*{Acknowledgments} 
Over a period of six years this paper evolved from a note not intended for publication 
to its present form. This subsection is appropriately long for a paper that took six years to write.  
My interest in the extender algebra was initiated by a conversation 
with John Steel in Oberwolfach in 2005 and motivated by 
(still open) Question~\ref{Q.Sigma-2-2}. I am indebted to   
 David Asper\'o,
Phillip Doebler,
Richard Ketchersid,  
Paul B. Larson,
Ernest Schimmerling,   
Ralf Schindler, 
John Steel, 
W. Hugh Woodin, 
Martin Zeman and 
Stuart Zoble, and especially 
Menachem Magidor and Grigor Sargsyan,  
for 
  explaining me the concepts presented below and pointing to flaws in  the earlier versions
 of this note. I would also like to thank the anonymous referee for a most detailed and useful 
 report.  I am  
 indebted to Grigor Sargsyan for convincing me that the construction of divergent 
 models of $\AD^+$ was within $\epsilon$ of the content of the present paper
  as of January 2011.\footnote{It wasn't.} 
  
I would like to thank John Steel and Hugh Woodin for their kind permission to include their
results in this note and to  Joan Bagaria and John Steel for their encouragement to
 publish this paper.
Part of this paper  
  was prepared during my stay at the Mittag--Leffler Institute in 
September 2009. I would like to thank the staff of the Institute for providing a pleasant and
stimulating atmosphere. 

The definition of the extender algebra and every uncredited result presented here 
is due to W.H. Woodin, but some of the proofs and formulations, as well as the exposition, 
 are mine. Mistakes, obscurities, and dwelling on trivialities are also all mine. 

\subsection*{Notation and terminology} Our notation is mostly standard but inner model theorist may want to take note
of two exceptions. First, $\pi$  usually denotes a transitive collapse instead of its inverse. 
Second, to an ultrafilter~$\cU$ we associate   quantifier $(\cU x)$ so that $(\cU x)\phi(x)$ 
stands for `the set of $x$ satisfying $\phi(x)$ belongs to $\cU$.'

All forcing considered here is set forcing. 

%
%

\section{An infinitary propositional logic} The extender algebra is the Lindenbaum algebra of certain theory in an infinitary propositional logic. 
The logic is described here and the theory will be described in \S\ref{S.EA}. 

\label{S.Logics}
\subsection{Logic $\CL_{\delta,\gamma}$} 
For regular cardinals $\gamma\leq \delta$ we shall define the 
infinitary propositional logic  $\CL_{\delta,\gamma}$. 
The interesting cases are $\gamma=\omega$ and $\gamma=\delta$, 
but it will be easier to develop the basic theory in the two cases parallelly. 
Let $\CLdg$   be the propositional logic with 
$\gamma$ variables~$a_\xi$, for $\xi< \gamma$, which in addition to the standard propositional 
connectives $\vee, \wedge, \rightarrow, \leftrightarrow$, and~$\lnot$ allows 
 infinitary conjunctions of the form   $\bigwedge_{\xi<\kappa} \varphi_\xi$ and 
 infinitary disjunctions of the form $\bigvee_{\xi<\kappa} \varphi_\xi$ 
for all $\kappa<\delta$.

In addition to the standard 
axioms and rules of inference for the finitary 
propositional logic, for each $\kappa< \delta$  and formulas $\varphi_\xi$, for $\xi<\kappa$
 the logic $\CLdg$ has  axioms $\vdash\bigvee_{\x<\kappa}\lnot
 \varphi_\xi\leftrightarrow\lnot\bigwedge_{\xi<\kappa}\varphi_\xi$ and 
 $\vdash \bigwedge_{\xi<\kappa} \varphi_\xi\rightarrow \varphi_\eta$, for every $\eta<\kappa$, 
 as well as  the  
 infinitary rule of inference: 
from $\vdash \varphi_\xi$ for all $\xi<\kappa$ infer $\vdash \bigwedge_{\xi<\kappa}\varphi_\xi$.
The provability relation for $\CLdg$ will be denoted by $\vdash_{\delta,\gamma}$
or simply $\vdash$ if $\delta$ and $\gamma$ are clear from the context. 
 Each proof in $\CL_{\delta,\gamma}$ is a well-founded tree and the assertion that~$\varphi$ is provable
in $\CL_{\delta,\gamma}$ is upwards  absolute between transitive models of ZFC. 
Since adding a new bounded subset of $\delta$ adds new formulas and new proofs to 
$\CL_{\delta,\gamma}$, it is not obvious that the assertion that $\varphi$ is not provable in $\CL_{\delta,\gamma}$ 
is upwards absolute between transitive models of ZFC. This is,  nevertheless, true: see Lemma~\ref{L1}.

Every $x\in \cP(\gamma)$ naturally defines a
 model  for $\CL_{\delta,\gamma}$ via $v_x(a_\xi)=\true$ iff $x(\xi)=1$ for $\xi<\gamma$. 
Define $A_\varphi=A_{\varphi,\delta,\gamma}$ via 
$$
A_\varphi=\set{x\in \cP(\gamma)}{ x\models \varphi}. 
$$
When $\gamma=\omega$ then these are the 
so-called \markdef{$\infty$-Borel} sets (see \eg, \cite[\S9.1]{Wo:Pmax2}).  Note that the sets of 
the form $A_{\varphi}$ for $\varphi\in \CL_{\omega_1,\omega}$ 
are exactly the Borel sets.\footnote{Keep in mind that, in spite of the connection 
with  Borel sets, this logic is  \markdef{not} the `usual'
 $L_{\omega_1,\omega}$ (\cite{Ke:Model}). Our $\CL_{\omega_1,\omega}$ 
happens  to be the propositional fragment of the latter. This is really an accident since 
in each of the two notations  `$\omega$' signifies
 a different constraint.}
 Also note that $x\models \varphi$ is absolute between transitive  models of ZFC
containing $x$ and~$\varphi$.

\subsection{Completeness of $\CL_{\delta,\gamma}$} 
The following two lemmas are standard. 

\begin{lem} 
\begin{enumerate}
\item $\vdash \varphi$ implies $\models \varphi$
for every formula $\varphi$ in $\CL_{\delta,\gamma}$. 
\item $\delta>2^\gamma$ implies that $\models\varphi$ but not $\vdash \varphi$
for some formula $\varphi$ in $\CL_{\delta,\gamma}$. 
\end{enumerate}
\end{lem} 

\begin{proof} Clause (1) can be proved by recursion on the rank of the proof. 

(2) For $x\subseteq \gamma$ let $\varphi_x=\bigvee_{\xi<\gamma} a_\xi^x$, 
where $a_\xi^x=a_\xi$ if $\xi\in x$ and $a_\xi^x=\lnot a_\xi$ if $\xi\notin x$. 
Then $y\models \varphi_x$ if and only if $y=x$, and therefore 
the formula $\varphi=\bigwedge_{x\subseteq \gamma} \lnot \varphi_x$
is not satisfiable. However, $\varphi$ is satisfiable 
in every forcing extension in which there exists a new 
subset of $\gamma$. 
\end{proof}

\begin{lem} \lbl{L1} For every $\varphi$ in $\CL_{\delta,\gamma}$ the following are equivalent. 
\begin{enumerate}
\item \label{L1.1} $\vdash \varphi$. 
\item \label{L1.2} $A_\varphi=\cP(\gamma)$ in all generic extensions. 
\item \label{L1.3} $A_\varphi=\cP(\gamma)$ in the extension by $\Coll(\omega,\kappa)$ for a large enough~$\kappa$. 
\end{enumerate}
\end{lem} 

\begin{proof} Clause \eqref{L1.1} is upwards absolute and by recursion on the rank of the proof it
easily implies \eqref{L1.2}. Also, \eqref{L1.2}   trivially implies \eqref{L1.3}.

Assume \eqref{L1.1} fails for $\varphi$ and let $\kappa$ be the cardinality of the set 
of  all  subformulas of $\varphi$.  
We may assume $a_\xi$ is a subformula 
of $\varphi$ if and only if $\xi<\kappa$. 

We claim that 
for any two formulas $\psi_1$ and $\psi_2$  in $\CL_{\delta,\gamma}$ 
such that $\psi_1\not\vdash \varphi$ we have either 
$\psi_1\land \psi_2\not\vdash \varphi$ or $\psi_1\land \lnot\psi_2\not\vdash \varphi$. 
Otherwise we have $\psi_1\vdash \psi_2\rightarrow \varphi$
and $\psi_1\vdash \lnot \psi_2\rightarrow \varphi$ and therefore $\psi_1\vdash \varphi$, a contradiction. 

In the extension by $\Coll(\omega,\kappa)$ enumerate all 
subformulas of $\varphi$ as $\psi_n$, for $n\in \omega$, and also enumerate
$\set{a_\xi}{ \xi<\kappa}$ as $b_n$, for $n\in\omega$. 
Recursively pick an increasing sequence of  
  $\CL_{\delta,\gamma}$-theories $\cT_n$, for $n\in\omega$, 
satisfying the following requirements: (i) $\cT_0=\{\lnot\varphi\}$ and each $\cT_n$ is a
 finite consistent set of ground-model formulas in $\CL_{\delta,\gamma}$. 
(ii) either  $b_n\in \cT_n$ or $\lnot b_n\in \cT_n$. 
(iii) if $\psi_n$ is of the form $\bigwedge_{\xi<\lambda} \sigma_\xi$ for some $\lambda$
then either $\psi_n\in \cT_n$ or $\lnot\sigma_\xi\in \cT_n$ for some $\xi<\lambda$. 
(iv) if $\psi_n$ is of the form $\bigvee_{\xi<\lambda} \sigma_\xi$ for some $\lambda$ then 
either $\lnot \psi_n\in \cT_n$ or $\sigma_\xi\in \cT_n$ for some $\xi<\lambda$. 

By the above, if $\cT_n$ satisfies the requirements then 
 $\cT_{n+1}$ as required can be chosen. 
After all $\cT_n$ have been chosen define $x\subseteq \gamma$ 
by $x=\set{\xi}{ a_\xi\in \bigcup_n \cT_n}$. 
Then $x\models \lnot\varphi$ can be proved by recursion on the rank of $\varphi$, 
using the fact that every infinitary conjunction and every infinitary disjunction 
appearing in  $\varphi$ is computed correctly. 
Therefore  in this extension \eqref{L1.3} fails. 
\end{proof} 


Recall that the \markdef{Lindenbaum algebra} of  a $\CL_{\delta,\gamma}$-theory $\cT$
is the Boolean algebra of all of equivalence classes of formulas in 
$\CL_{\delta,\gamma}$ with respect to the equivalence 
relation~$\sim_{\cT}$ defined by $\varphi\sim_{\cT}\psi$ if and 
only if $\cT\vdash \varphi\leftrightarrow \psi$. 
We shall denote this algebra 
by  $\Bdg/\cT$ and 
consider its positive elements  as a forcing 
notion (see Lemma~\ref{L3} below). 

\begin{lem}\lbl{L2} For every $\CLdg$-theory $\cT$
 such that  $\Bdg/\cT$ has the $\delta$-chain condition  
$\Bdg/\cT$ is a complete Boolean algebra. 
\end{lem}

\begin{proof} Immediate, since both $\Bdg$ and $\cT$ are $\delta$-complete. 
\end{proof} 

For $x\subseteq \gamma$ such that $x\models \cT$ 
define an ultrafilter of $\Bdg/\cT$ by (here $[\varphi]$ stands for the equivalence class of $\varphi$)
\[
\Gamma_x=\set{[\varphi]\in \Bdg/\cT}{ x\models \varphi}. 
\]
Note that for every generic $\Gamma\subseteq \Bdg/\cT$ there is 
the unique $x\subseteq \gamma$ such that $\Gamma_x=\Gamma$, 
defined as $x=\set{\xi}{ a_\xi\in \Gamma}$. 

\begin{lem} \lbl{L3} Assume $\cT$ is a theory in $\CLdg$ and
 $M$ is a transitive model of a large enough fragment of ZFC 
such that  $\{\delta,\gamma,\cT\}\subseteq M$, 
and $\Bdg/\cT$ has the $\delta$-chain condition in $M$. 
Then for every $x\in \cP(\gamma)$ we have  $x\models \cT^M$ if and only if 
$x$ is  $\Bdg/\cT$-generic over $M$. 
\end{lem}

\begin{proof} If $x$ is $\Bdg/\cT$-generic over $M$ then a proof by induction on the complexity  shows that  $x\models\varphi$ for all $[\varphi]\in \Gamma$, hence
$\set{\varphi\in\CLdg}{ x\models \varphi}$ includes $\cT^M$.

Now assume $x\models \cT^M$. 
Let $\varphi_\xi$ ($\xi<\kappa$) be a maximal antichain of $(\Bdg/\cT)^M$ 
that belongs to $M$.
By the $\delta$-chain condition $\kappa<\delta$ and therefore $\bigvee_{\xi<\kappa}\varphi_\xi$
is in $\CLdg^M$. By the maximality of the antichain we have $\cT^M\models\bigvee_{\xi<\kappa}\varphi_\xi$. 
Therefore $x\models \bigvee_{\xi<\kappa} \varphi_\xi$. This implies 
$x\models \varphi_\xi$ for some $\xi<\kappa$ and $\Gamma_x$ intersects the given maximal 
antichain. Since the antichain was arbitrary, we conclude $x$ is generic. 
\end{proof}


\section{Elementary embeddings} In the present section we recall some basic facts about elementary 
embeddings and extenders. 
\label{S.EE} 
\subsection{Extenders I} 
\lbl{SS1} 
We only sketch the bare minimum of the theory of extenders. 
For more details see \cite{martin94iterationtrees} or~\cite[\S26]{Kana:Book}.  
 An \markdef{extender} $E$ is a set that codes an elementary 
embedding $j_E\colon\V\to M$. 
For every elementary embedding $j\colon\V\to M$ such that 
with $\kappa=\crit(j)$ we have $\kappa<\lambda<j(\kappa)$ 
 there is an extender $E$ in $\V_{\lambda+1}$ such that $j_E$ and 
$j$ coincide up to $\V_\lambda$ in the sense 
that $j(A)\cap \V_\lambda=j_E(A)\cap \V_\lambda $ for all $A\subseteq \V_\kappa$. 
The model $M$ is constructed as a direct limit of ultrapowers of $V$ and it is 
denoted by $\Ult(\V,E)$. 

All  extenders $E$ used in this paper are such that  $j_E(\kappa)>\lambda$, where $\kappa=\crit(E)$ and $\lambda$ is the strength of $E$. Such extenders are called \markdef{short} extenders. Long extenders are needed to describe stronger large cardinal embeddings.

A \markdef{generator} of an elementary embedding $j\colon\V\to M$ is an ordinal $\xi$ such that
there are an inner model $N$ and elementary embeddings $i_1$ and $i_2$ 
such that the diagram 
\[
\diagram
V\rto^{i_1}\drto_j & N\dto^{i_2}\\
& M
\enddiagram 
\]
commutes and $\crit(i_2)=\xi$. 
For example, the critical point $\kappa$ is the least generator and a counting argument shows that 
 an elementary embedding such that
$j(\kappa)\geq (2^\kappa)^+$ must have other generators. 
A \markdef{generator} of an extender $E$ is a generator of $j_E$. 
The \markdef{strength} of  $E$ is the largest $\lambda$ such 
that $\V_\lambda\subseteq M$, where $j_E\colon\V\to M$.

The only properties of extenders used in the present paper are (E1) and~(E2). 
\begin{enumerate}
\item [(E1)] 
If $E$ is 
an extender with $\kappa=\crit(j_E)$ then $\Ult(N,E)$ can be formed for every model $N$ such that 
$(\V_{\kappa+1})^N=\V_{\kappa+1}$. Also, $\Ult(N,E)\supseteq \V_\lambda$, where $\lambda$ is 
the strength of $E$ and whether or not $\xi$ is a generator of~$E$ depends only on 
$E$ and $\V_{\kappa+1}$. 
\end{enumerate}
For a proof 
of (E1)  see \cite[Lemma~1.5]{martin94iterationtrees}, or note that this is 
immediate
from Definition~\ref{D.Ext}  below since all ultrafilters $E_s$ concentrate on $[\kappa]^{<\omega}$. 
While (E1) is a property of all extenders, (E2) below is not. However, we shall consider only the 
extenders satisfying (E2). 
\begin{enumerate}
\item [(E2)]  The strength of $E$ is greater than 
the supremum of the generators of~$E$. 
\end{enumerate}
It is important to note that the strength depends only on $E$,
and not on the model to which $E$ was applied.
In particular, if $E$ is an extender in $M$ with 
$\kappa=\crit(j_E)$, $\lambda$ is the strength of $E$,
  and $M\cap \V_{\kappa+1}=N\cap \V_{\kappa+1}$ then we have
$\Ult(M,E)\cap \V_\lambda=\Ult(N,E)\cap \V_\lambda$.

\subsection{Woodin cardinals} 
If $A$ is a set such that $j\colon\V\to M$ satisfies
$\V_\lambda=(\V_\lambda)^M$ and $j(A)\cap \V_\lambda=A\cap \V_\lambda$ then we say $j$ is an \markdef{$A,\lambda$-strong}
embedding. This may differ from the standard definition of a $\lambda$-strong embedding
but I will consider only $\lambda$ such that $\kappa+\lambda=\lambda$ for $\kappa=\crit(j)$, in 
which case there is no difference.   If $j_E$ is an $A,\lambda$-strong then we say $E$ is $A,\lambda$-strong. 
A cardinal $\delta$ is a \markdef{Woodin cardinal} if for every $A\subseteq \V_\delta$
there is $\kappa<\delta$ such that there are $A,\lambda$-strong elementary embeddings with critical 
point $\kappa$ for an arbitrarily large $\lambda<\delta$. We say that $A$ \markdef{reflects} to $\kappa$. 
Note that the Woodinness of~$\delta$ is witnessed by the extenders in $\V_\delta$, 
and therefore $\delta$ is Woodin in $V$ if and only if it is Woodin in $\LL(\V_\delta)$. 
Moreover, it suffices to consider only the extenders that satisfy property~(E2). 

A cardinal $\kappa$ is \markdef{$\lambda$-strong} if it is a critical point of an  $\emptyset,\lambda$-strong elementary embedding $j\colon\V\to M$.  In particular every Woodin cardinal $\delta$ is a 
limit of cardinals each of which is  $\lambda$-strong for all $\lambda<\delta$. 



\subsection{Extenders II}
A reader not  interested in  extenders per se  may  want to  skip the rest of this section
on the first 
reading and take (E1) and (E2) for granted. 
The actual definition of an extender  is, strictly speaking,  not necessary 
  for our present purpose. However,  this notion is central\footnote{after all, it is the \markdef{extender} algebra} in the theory and we  include it
for the reader's convenience. 
Every extender is of the form as described in Example~\ref{Ex.1}. 

\begin{exa} \label{Ex.1} 
Assume $j\colon\V\to M$ is an elementary embedding with $\crit(j)=\kappa$. 
Fix $\lambda$ such that $\kappa\leq \lambda<j(\kappa)$. 
Typically, we take $\lambda$ such that $M\supseteq \V_\lambda$. 
For $s\in [\lambda]^{<\omega}$ 
define $E_s\subseteq [\kappa]^m$ (where $m=|s|$) by  
\[
X\in E_s\text{ if and only if } s\in j(X). 
\]
Then $E(j,\lambda)=
\langle E_s\colon s\in [\lambda]^{<\omega}\rangle$ is a $(\kappa,\lambda)$-extender. 
Ordinal $\lambda$  is  called the \markdef{length} of the extender $E$ and is denoted 
by $\lh(E)$. In all of our applications $\lh(E)$ will  equal the strength of $E$. 
\end{exa}

The following facts can be found \eg, in \cite{steel06innermodel}. 
For an extender $E$ we write  
\[
\kappa_E=\crit(E)
\]
 (the \markdef{critical 
point} of $E$, \ie, the least ordinal moved by $j_E$) 
and 
\[
\lambda_E=\sup\set{\eta}{ \V_\eta\subseteq (\V_\eta)^M}
\]
(the \markdef{strength} of $E$). 
Note that if  $j\colon\V\to M$ is an elementary embedding such that $\crit(j)=\kappa$, 
 the strength of $j$ is $\lambda$, and $\lambda$ is a strong limit cardinal, 
 then the $(\kappa,\lambda)$ extender $E$ defined 
in Example~\ref{Ex.1} 
satisfies $\Ult(\V,E)\cap \V_\lambda=M\cap \V_\lambda$ 
and $j(A)\cap \V_\lambda=j_E(A)\cap \V_\lambda$
for all~$A$ and is therefore $\lambda$-strong.  
The assumption that $\lambda$ is strong limit is needed to code subsets of $\cP(\alpha)$ by sets 
of ordinals for every $\alpha<\lambda$.

Fix a cardinal $\kappa$. For $m\in \omega$ and $s\subseteq m$ with $|s|=n$ 
 consider the projection map 
$\pi=\pi_{m,s}\colon [\kappa]^m\to [\kappa]^s$ defined by 
\[
\pi(\langle \xi_i\colon i<m\rangle)=\langle \xi_i\colon i\in s\rangle.
\]
More generally, if $s\subseteq t$ are finite sets of ordinals 
(listed in the increasing order) 
then the projection $\pi=\pi_{t,s}\colon [\kappa]^t\to [\kappa]^s$
is defined by 
\[
\pi(\langle \xi_i\colon i\in t\rangle)=\langle \xi_i\colon i\in s\rangle.
\]
Recall that if $\cU$ and $\cV$ are ultrafilters on sets $I$ and $J$, respectively, 
then we write $\cU\lRK \cV$ if and only if there is $h\colon J\to I$ such that 
\[
X\in \cU\text{ if and only if } h^{-1}(X)\in \cV. 
\]
In this situation we say \markdef{$\cU$ is Rudin--Keisler reducible to $\cV$}
and that $h$ is the \markdef{Rudin--Keisler reduction of $\cU$ to $\cV$}. 

Respecting the notation commonly accepted in the theory of large cardinals, 
we denote 
an ultrapower of a structure $M$ associated to an  ultrafilter $\cU$ 
by $\Ult(M,\cU)$. Its elements are the equivalence classes of $f\in M^I\cap M$, 
\[
[f]_{\cU}=\set{g\in M^I\cap M}{ (\cU i) f(i)=g(i)}
\]
and the membership relation is defined by $[f]_{\cU} \in [g]_{\cU}$ if and only if
$(\cU i) f(i)\in g(i)$. 

If $\cU$ is $\aleph_1$-complete then $\Ult(M,\cU)$ is well-founded whenever $M$ is well-founded, 
and we identify $\Ult(M,\cU)$ with its transitive collapse. 

Assume $\cU$ and $\cV$ are ultrafilters on index-sets $I$ and $J$, respectively, 
and $\cU\lRK \cV$ is witnessed by a reduction $h\colon J\to I$. 
Then for any structure $M$ we can define a map $j_h\colon M^I/\cU  \to M^J/\cV$ by
\[
j_h([f]_{\cU})=[f\circ h]_{\cV}
\]
In the following lemma $M$ is any structure,   $\cU$ and $\cV$ 
are arbitrary ultrafilters,  and the proof is   straightforward.

\begin{lem} \label{L.RK} 
If $\cU\lRK \cV$,  $h$ is the Rudin--Keisler reduction, and $j_h$ is defined as above, 
then the diagram 
\[
\diagram
M \rto^{j_{\cU}}\drto_{j_{\cV}} & \Ult(M,\cU)\dto^{j_h}\\
& \Ult(M,\cV)
\enddiagram
\]
commutes and $j_h$ is an elementary embedding. \qed
\end{lem}

For a finite set $s$ and $i<|s|$ let $s_i$ denote its $i$-th element.  As common in 
set theory, we start counting at $0$. 

\begin{dfn} \label{D.Ext} 
Assume $\kappa<\lambda$ are uncountable cardinals. A \markdef{$(\kappa,\lambda)$-extender} is 
 $E\colon [\lambda]^{<\omega}\to \V_{\kappa+2}$ such that for all $s$ and $t$ in $[\lambda]^{<\omega}$ we have
\begin{enumerate}
\item [(a)] $E_s$ is a nonprincipal  $\kappa$-complete ultrafilter on $[\kappa]^{|s|}$.  
\item [(b)] If $s\subseteq t$ then $\pi_{t,s}$ is a Rudin--Keisler reduction of $E_s$ to $E_t$. 
\item [(c)] \markdef{Normality}: If $s\in [\lambda]^{<\omega}$, $i<|s|$,  and 
$f\colon [\kappa]^{|s|}\to \kappa$ is such that $f(u)<u_i$ 
for $E_s$ many~$u$, then 
there  exist $\xi<s_i$ and $j$ such that $f\circ \pi_{a\cup \{\xi\},a}(u)=u_j$
for $E_{s\cup \{\xi\}}$ many $u$. 

\item [(d)] \markdef{Countable completeness}: if  $s(n)\in[\lambda]^{<\omega}$ 
 and $X(n)\in E_{s(n)}$ for all $n<\omega$ then 
there is increasing $h\colon \bigcup_n s(n)\to\kappa$
such that $h''s(n)\in X(n)$ for all~$n$. 
\end{enumerate}
\end{dfn} 

Note that, with the above notation for a $(\kappa,\lambda)$-extender, 
we have $\kappa_E=\kappa$ but not necessarily $\lambda_E=\lambda$. 
However, 
the length and strength of each extender in all of our applications 
will coincide. 

 If $E$ is an extender, then the models 
\[
M_s=\Ult(\V,E_s)
\]
are, by the $\kappa$-completeness of $E_s$,  well-founded. 
By Lemma~\ref{L.RK}, they form a directed system under the embeddings 
(writing $j_{t,s}$ for $j_{\pi_{t,s}}$)
\[
j_{t,s}\colon M_t \to M_s. 
\]
The direct limit of this system will be denoted by $\Ult(\V,E)$ and identified with its transitive collapse
if it is well-founded. 

It is not difficult to show that if $E$  satisfies (a), (b) and (c) of Definition~\ref{D.Ext} 
then the corresponding ultrapower $\Ult(\V,E)$ is well-founded if and only if $E$ satisfies (d)
as well.

If $E$ is a $(\kappa,\lambda)$-extender, let $M$ be a model to which $E$ can be applied 
 and let  $\kappa<\xi\leq \lambda$. 
One defines $E\rs \xi=\langle E_s\colon s\in [\xi]^{<\omega}\rangle$. 
Like in Lemma~\ref{L.RK} 
 one can define an elementary embedding $i$ such that the diagram  
\[
\diagram
M\rrto^{j_{E\rs \xi}} \drrto_{j_E} && \Ult(M,E\rs \xi)\dto^i\\
&& \Ult(M,E)
\enddiagram
\]
commutes. If $\crit(i)=\xi$ then we say $\xi$ is a \markdef{generator} of $E$.

\begin{lem} If $j\colon\V\to M$ is a $\lambda$-strong embedding with $\crit(j)=\kappa$. 
Let $E=E(j,\lambda)$ be as in Example~\ref{Ex.1}. Then 
\begin{enumerate}
\item $E$ satisfies (E1) and (E2). 
\item For every $X\subseteq \kappa$ we have $j(X)\cap \lambda=j_E(X)\cap \lambda$. 
\end{enumerate}
In particular, if $\delta$ is a Woodin cardinal then there is a family $\vec E\subseteq \V_\delta$
of extenders satisfying (E1) and (E2) such that $\delta$ is Woodin in $\LL[\vec E]$. 
\end{lem} 

\begin{proof} 	Clause (E1) is a consequence of the
fact that $\Ult(N,E)$ depends only on $\V_{\kappa+1}\cap N$. 
Clearly every generator of a $(\kappa,\lambda)$-extender
is $\leq \lambda$, and therefore ~(E2) follows. 
\end{proof}

Lemma~\ref{L.E4} below  will not be needed elsewhere in the present paper.  
It is included only as
an illustration that 
the countable completeness of the extenders is a necessary requirement for 
wellfoundedness of the ultrapower. 

\begin{lem} \label{L.E4} Assume $\kappa$ is a  measurable cardinal. 
Then there is $\lambda>\kappa$ and $E\colon [\lambda]^{<\omega}\to \V_{\kappa+2}$ such that
\begin{enumerate}
\item [(a)] $E_s$ is a nonprincipal  $\kappa$-complete ultrafilter on $\kappa^s$,  
\item [(b)] If $s\subseteq t$ then $\pi_{t,s}$ is a Rudin--Keisler reduction of $E_s$ to $E_t$
\end{enumerate}
 such that the ultrapower  $\prod_E V$ is ill-founded. 
\end{lem} 

\begin{proof}  Let $j\colon\V\to M$ be an elementary embedding with $\crit(j)=\kappa$. 
Let $\lambda=\sup\set{n\in \omega}{j^n(\kappa)}$. 
For all $X$, $n$ and $k$ we have the following. 
\[
j^{n+k}(X)\cap j^k(\kappa)=j^k(j^n(X)\cap \kappa)=j^k(X\cap \kappa)=j^k(X)\cap j^k(\kappa). 
\]
Because of this for a finite $s\subseteq \lambda$ the following defines
 a subset $E_s$ of   $\kappa^{|s|}$: 
\[
X\in E_s\text{ if and only if } s\in j^n(X)
\]
where $n$ is  
such that  $\max(s)<j^n(\kappa)$. 
Then $E_s$ is clearly a $\kappa$-complete ultrafilter. 

The map $E$,  $[\lambda]^{<\omega}\ni s\mapsto E_s\in \V_{\kappa+1}$ satisfies
(E1) and (E2). We can therefore form the direct limit of the ultrapowers
$\prod_{E_s} V$, $s\in [\lambda]^{<\omega}$. 
However, this ultrapower is not well-founded. Let $s(n)=\set{j^i(\kappa)}{ i<n}$
and $M_n=\Ult(\V,E_{s(n)})$. Then $j_{s(n),s(n+1)}(\kappa)>\kappa$, and therefore 
in the direct limit we have a decreasing $\omega$-sequence of ordinals.  
\end{proof} 

The problem with $E$ defined in Lemma~\ref{L.E4} is that the
ultraproducts are iterated the `wrong way.'  Let us consider this
example a little more closely. If $s(n)=\set{j^m(\kappa)}{ m<n}$
then $E_{s(n)}$ is the set of all $X\subseteq \kappa^n$ such that
\[
(\cU \xi_0)(\cU \xi_1)\dots (\cU\xi_{n-1}) \langle \xi_{n-1},\dots,\xi_1, \xi_0\rangle\in X
\]
Then the set of all decreasing $n$-tuples of ordinals $<\kappa$
belongs to $ E_{s(n)}$  for each $n$.

\section{The extender algebra}\label{S.EA} 
Assume $\vec E$ is a family of extenders in $\V_{\delta+2}$. 
Typically,~$\vec E$ will be  a subset of $\V_\delta$ and it will  witness  that 
  $\delta$ is a Woodin cardinal. 
Let $\cTdg(\vec E)$ be the deductive closure in $\CL_{\delta,\gamma}$ of all sentences of the form
\begin{enumerate}
\item []$\Psi(\vec\varphi,\kappa,\lambda)$:\quad
$\bigvee_{\xi<\kappa} \varphi_\xi\leftrightarrow \bigvee_{\xi<\lambda}\varphi_\xi$
\end{enumerate} 
for a sequence $\vec\varphi=\langle\varphi_\xi\mid \xi<\delta\rangle$ 
in $\CL_{\delta,\gamma}$ 
such that $\varphi_\xi\in \V_\kappa$ for all $\xi<\kappa$, 
ordinals $\kappa<\lambda$, and  
extender $\vec E$  with $\crit(E)=\kappa$ that is $\vec \phi,\lambda$-strong.\footnote{Early versions of this paper contained a nonstandard definition of theory $\cTdg(\vec E)$.  
I~have decided to adopt the standard, more flexible,  definition.}

\begin{lem}\label{L.reflect+}
If $E\in \vec E$ and  $\crit(E)=\kappa$ then for 
every $f\colon \kappa\to \CL_{\delta,\gamma}\cap \V_\kappa$ we have
$
\CL_{\delta,\gamma}\vdash \bigvee_{\xi<\kappa} f(\xi)\lthen j_E (f)(\kappa)$. 
\end{lem} 

\begin{proof} If $\lambda$ is the strength of $E$ then with $\vec\varphi=\langle j_E (f)(\xi): \xi<\lambda)\rangle$ we have that $E$ is is $\vec\varphi,\lambda$-strong and $\Psi(\vec\varphi,\kappa,\lambda)$ implies the above formula. 
\end{proof} 


In the following it may be worth emphasizing that $x$ is assumed to 
belong to the same inner model as $\vec E$
(\compare{} Theorem~\ref{T1}).  

\begin{lem} For every real $x$ in $\LL[\vec E]$ 
we have $x\models \cT_{\delta,\omega}(\vec E)$. In 
particular,  $\cT_{\delta,\omega}(\vec E)$ is a consistent theory. 
\end{lem} 

\begin{proof} 
Fix a sequence $\vec\varphi$ that reflects to $\kappa$ and this is witnessed by 
extenders in $\vec E$. 
We need to check $\Psi(\vec\varphi,\kappa,\lambda)$ for all~$\lambda>\kappa$. 
We may assume $x\models\varphi_\xi$ for some $\xi$, since otherwise
$\Psi(\vec\varphi,\kappa,\lambda)$ vacuously 
holds for all $\lambda$. Pick an extender $E\in \vec E$
such that $j_E$ is $(\vec \varphi,\lambda)$-strong for some $\lambda>\xi$. 
Since $x$ is a real it is not moved by 
any~$j$ and therefore by elementarity we have $x\models \varphi_\eta$ for some 
$\eta<\kappa$ and the conclusion follows. 
\end{proof} 
 
 Note that if a cardinal 
  $\gamma$ is $\delta$-strong (or equivalently, in the terminology introduced above, 
if $\delta$ reflects to $\gamma$) and this is witnessed by extenders in~$\vec E$ then 
 $\gamma\not\models \cT_{\delta,\gamma}(\vec E)$, as can be seen by taking $\varphi_\xi(x)$ 
 to be $a_\xi$ if $\xi\in x$ and $\lnot a_\xi$ if $\xi\notin x$.  
 However, if $\crit(E)=\gamma$ then by the mininality of $\gamma$ and elementarity 
 we have 
 $\gamma\models j_E(\cT_{\delta\gamma}(\vec E))$

The \markdef{extender algebra} with $\gamma$ generators  corresponding to $\vec E$ is the algebra
\[
\cWdg(\vec E)=\Bdg/\cTdg(\vec E).
\]
Most important instances of $\cWdg(\vec E)$  are given by 
 $\gamma=\omega$ and~$\gamma=\delta$ but for convenience 
 we develop the theory of $\cWdg(\vec E)$
 for an arbitrary~$\gamma\leq \delta$.

\begin{lem} \lbl{L4} If $\delta$ is a Woodin cardinal and $\vec E$ is a system of extenders
witnessing its Woodinness, then  
$\cWdg$ 
 has the $\delta$-chain condition and is therefore complete.  
\end{lem}

\begin{proof} 
Pick a sequence $\vec\varphi=\set{\varphi_\xi}{ \xi<\delta}$ in $\CLdg$. 
We want to prove that 
$\set{[\varphi_\xi]}{ \xi<\delta}$ is not an antichain in $\cWdg$. 
The set 
\[
\bfC=\set{\kappa<\delta}{ (\forall \xi<\kappa) \varphi_\xi\in \V_\kappa}
\] 
is a club in $\delta$. Using some reasonable coding of pairs in~$\V_\delta$ by elements of~$\V_\delta$
find $X\subseteq \V_\delta$ that codes the pair $\bfC,\vec\varphi$. 
Since $\delta$ is Woodin  there exists an  $X,\lambda$-strong extender $E\in \vec E$ 
such that with $\kappa=\crit(E)$  the set  $[\kappa,\lambda)\cap \bfC$ is nonempty. 
By the elementarity of $j_E$  this implies  $\kappa\in \bfC$. 
Therefore Lemma~\ref{L.reflect+} applies, and $[\varphi_\kappa]$ is compatible 
with~$[\varphi_\xi]$ for some~$\xi<\kappa$. 
 
The completeness of $\cWdg$ now follows by Lemma~\ref{L2}.
\end{proof} 

A partial  converse of Lemma~\ref{L4}, 
 that if $\cWdg(\vec E)$ has $\delta$-cc and $\V_\delta^\#$ exists then $\delta$ is Woodin, was proved in 
\cite{KeZo:On} and independently by Woodin.  It is not known whether the converse of Lemma~\ref{L4} is true.

The following lemma, which will play an important role in  
\S\ref{S.Div}, was proved in \cite[Lemma~3.6]{Hj:Some} in the case when $\lambda=2$. 
The general case requires no new ideas. 

\begin{lem}[Hjorth] \label{L4+}
if $\lambda<\delta$ then product of 
$\lambda$ copies of $\cWdg$ has the $\delta$-chain condition. 
\end{lem}

\begin{proof} 
Assume the contrary and let $A=\set{\langle [\varphi_\xi(i)]}{ i<\lambda\rangle \mid \xi<\delta}$ be an antichain
in $(\cWdg)^\lambda$. By argument as in the proof of Lemma~\ref{L4} we can 
find $\kappa$ such that $\varphi_\xi(i)\in \V_\kappa$ for all $i$ and all $\xi<\kappa$
and there is an $A,\lambda$-strong extender $E$ with critical point $\kappa$ ($\lambda=\kappa+1$ 
suffices).

Then  $\cU=\set{X\subseteq \kappa}{ j_E(X)\ni \kappa}$
 is a normal $\kappa$-complete ultrafilter on $\kappa$. 
If there were $\xi<\eta<\kappa$ such that $[\varphi_\xi(i)]$ and $[\varphi_\eta(i)]$ 
are compatible for all $i<\lambda$, then $A$ would not be an antichain. 

Consider the map  $f\colon [\kappa]^2\to \lambda$ defined by 
\[
f(\{\xi,\eta\})=\min\set{i<\lambda}{[\varphi_\xi(i)] \text{ and $ [\varphi_\eta(i)]$ are incompatible}}. 
\]
Since $\cU$ is a normal ultrafilter we have $\cU\to (\cU)^2_\lambda$ 
and  there are $X\in \cU$ and $i<\lambda$ such that $f$ has constant value $i$ on $[X]^2$.   
Let 
\[
\psi_\xi=\begin{cases} 
\varphi_\xi(i),\text{ if $\xi\in X$ or $\xi\geq \kappa$}\\
\bot, \text{ if $\xi\notin X$ and $\xi<\kappa$}
\end{cases} 
\]
Then $\psi_\kappa= j(\vec\psi)_\kappa$.  \L os's theorem implies that $[\psi_\kappa]$ is 
incompatible with $[\psi_\xi]$ for all $\xi<\kappa$. However, 
$\psi_\kappa=j(\vec\psi)_\kappa$ and this  contradicts    Lemma~\ref{L.reflect+}. 
 \end{proof} 


The following curious fact  (not needed in  proofs of the main results of this note) 
was pointed out to me by Paul Larson.

\begin{lem} \label{L.universal} 
Assume 
$\gamma$ is any cardinal less than the least critical point of each extender in $\vec E$
and $\bbP$ is a forcing notion of cardinality $\gamma$. 
Then in~$\cWdg(\vec E)$ there is a condition $p$ that forces that 
$\bbP$ is a regular subordering of $\cWdg(\vec E)$.  
\end{lem}

\begin{proof}[The first proof of Lemma~\ref{L.universal}]
Pick a bijection between $\bbP$ and $\gamma$ and hence identify $\bbP$ 
with $\langle\gamma,\leq_{\bbP}\rangle$ for some partial ordering $\leq_{\bbP}$ on $\gamma$. 
   Let the  sentence $\varphi$ be the  conjunction of 
axioms expressing the following. 
\begin{enumerate}
\item [(a)] The order on $\bbP$: $a_\xi\lthen a_\eta$ whenever $\xi\leq_{\bbP}\eta$, 
\item [(b)] The incompatilibity relation on $\bbP$: 
$\bigwedge_{\zeta<\gamma} 
(\lnot (a_\zeta\lthen a_\eta)\lor \lnot (a_\zeta\lthen a_\xi))$, if $\xi\perp_{\bbP}\xi$.  
\item [(c)] For every maximal antichain $\cA$ of $\bbP$, $\bigvee_{\xi\in \cA} a_\xi$  .
\end{enumerate}
Then $\varphi$ is in $\CLdg$ and since its size is below the least critical point of an extender 
in $\vec E$, it is consistent with $\cT_{\delta\gamma}(\vec E)$. 
Therefore we may take $p$ to be (the equivalence class of) $\varphi$.
\end{proof} 

We shall give another proof of Lemma~\ref{L.universal} after  Theorem~\ref{T1.1.1}.  

\subsection{Iteration trees} 
We say that $T=(\zeta,\leq_T)$ is a \markdef{tree order} on an ordinal $\zeta$ if 
\begin{enumerate}
\item $T=(\zeta,\leq_T)$ is a tree with root $0$, 
\item $\leq_T$ is  coarser than $\leq$,  
\item every successor ordinal is a successor in $\leq_T$, 
\item\label{IT.4}  if $\xi\leq\zeta$
 is a limit ordinal then the set of $\leq_T$-predecessors of $\xi$ is cofinal 
in~$\xi$. 
\pushcounter
\end{enumerate} 
Consider a transitive model $M$ of a large enough fragment of ZFC and a system $\vec E$ of 
extenders in $M$. We allow $M$ to be a proper class. As a matter of fact, it is typically 
going to be a proper class. Nevertheless, we omit the (straightforward) 
nuisances involved in  
formalization of the notion of an iteration tree in ZFC.

An $\vec E$-\markdef{iteration tree} is a structure consisting of 
 $\langle T, M_\eta, E_\xi\mid \eta\leq\zeta, \xi<\zeta\rangle$, 
together with a commuting system of elementary embeddings $j_{\xi\eta}\colon M_\xi\to M_\eta$ for $\xi \leq_T \eta$ such that 
\begin{enumerate}
\popcounter
\item $\leq_T$ is a tree order on $\zeta$, 
\item $E_\xi$ is an extender in ${\vec E}^{M_\xi}$, 
\item\label{IT.7}  If $\kappa=\crit(E_\xi)$ then the 
immediate $\leq_T$-predecessor of $\xi+1$ is the least ordinal 
$\eta$ such that $M_\eta\cap \V_{\kappa+1}=M_\xi\cap \V_{\kappa+1}$, 
\item \label{IT.8} If $\xi+1$ is the immediate $\leq_T$ successor of $\eta$ 
then $M_{\xi+1}=\Ult(M_\eta,E_\xi)$, hence $j_{\eta,\xi+1}=j_{E_\xi}$ as computed with respect 
to $M_\eta$, 
\item \label{IT.8+} If $\xi$ is a limit ordinal then $M_\xi$ is the direct limit of 
$M_\eta$, for $\eta<_T \xi$, and
\item each $M_\xi$ is well-founded. 
\item \label{I.IT.11} If $\xi<\eta$ and $\zeta$ is the length of $E_\xi$ then 
$M_\xi\cap \V_\zeta= M_\eta \cap \V_\zeta$. 
\pushcounter
\end{enumerate}
It is usually not required that an iteration tree satisfies condition \eqref{IT.7}, 
and the iteration trees satisfying this condition are called \markdef{normal} iteration
trees. A straightforward induction shows that condition \eqref{I.IT.11} is a consequence of the previous conditions. 

If $E_\xi$ was always applied to $M_\xi$, then we would have a \markdef{linear} iteration
that is moreover \markdef{internal}---\ie, each extender used in the construction
belongs  to the model to which it is applied. The wellfoundedness of such an iteration
follows from a rather mild additional condition about the extenders. 
On the other hand, the choice of condition~\eqref{IT.7} is behind the 
power of the iteration trees (see the proof of Theorem~\ref{T1}). 
This condition also prevents the obstacle to wellfoundedness of the iteration exposed in 
Lemma~\ref{L.E4} (see Lemma~\ref{L.no-overlaps}). 
It will be important that the extenders in $\vec E$ have the property (E2) (see \S\ref{SS1}), that 
for every generator $\xi$ of $E$ we have $\Ult(\V,E)\cap \V_\xi=\V_\xi$, or in other
words, that every extender  is $\xi$-strong for each of its generators~$\xi$. 

\begin{lem} \label{L.no-overlaps} Assume  $\langle T, M_\eta, E_\xi\mid \eta\leq\zeta,
\xi<\zeta\rangle$, 
 is an iteration tree such that every 
extender used in its construction  satisfies (E2). 
Assume $E_0$ and $E_1$ are extenders used along the same branch of $T$ and $E_1$ 
was used after~$E_0$. 
Then $\kappa=\crit(E_1)$ is greater than the supremum of all generators of~$E_0$. 
\end{lem} 

\begin{proof} 
Assume the contrary. Therefore $\kappa$ is less or equal than some generator of $E_0$, 
and  (E2) implies that $\kappa$ is smaller than  the strength of $E_0$. 
Let us first consider the case when 
 $E_1$ was applied to 
$M_\beta=\Ult(M_\alpha,E_0)$ for some $\alpha$ along the branch. 
If $\xi$ is the strength of $E_0$ then we have  
$M_\beta\cap \V_\xi=M_\alpha\cap \V_\xi$. 
Therefore, since $\kappa<\xi$ and  $E_1$  could be applied to $M_\beta$, 
 it could be applied to $M_\alpha$ as well. 
 This  contradicts our assumption that $E_1$ was applied to $M_\beta$. 
 
We may therefore assume $E_1$ was applied to a model $M_\gamma$ 
that is a direct limit of other models 
on the branch, including $\Ult(M_\alpha,E_0)$ for some $\alpha$. The above and an  
induction argument show that 
$\crit(j_{E_1})$ is greater than any generator of any extender used in the construction of this
branch, including the generators of $E_0$. 
 \end{proof} 

Note that an iteration tree
  $\langle T, M_\eta, E_\xi\mid\eta\leq\omega_1, \xi<\omega_1\rangle$ 
has a branch of length~$\omega_1$ by~\eqref{IT.4}. 
In Lemma~\ref{L.branch} below, and elsewhere, we assume that all critical points
of elementary embeddings $j_{\xi\eta}$  used in building 
 an iteration tree are countable ordinals. 
 Lemma~\ref{L.branch} is an attempt to extract one of the key ideas from the proof of
 Theorem~\ref{T1}. 
It is part of the proof of the comparison lemma for mice,
which at the level of generality we are dealing in now, is due
to Martin and Steel \cite{martin94iterationtrees}. 
 Preprint 
\cite{Schi:Notes} was very helpful during  the extraction of this lemma.

\begin{lem}\label{L.branch}
Assume  $\langle T, M_\eta, E_\xi\mid \eta\leq\omega_1,\xi<\omega_1\rangle$ is an iteration tree
and $b\subseteq \omega_1$ is its cofinal branch. Assume 
$H\prec H_{(2^{\aleph_1})^+}$ is countable
and it contains the iteration tree. Let  $\bar H$ be  its transitive collapse
and let $\pi\colon H\to \bar H$ be the collapsing map. 

With $\alpha=H\cap \omega_1$ we have the following. 
\begin{enumerate}
\popcounter
\item \label{L.branch.0}  $\alpha\in b$ and
$M_\alpha$ is the direct limit of $M_\xi$, $\xi\in b\cap \alpha$. 
\item\label{L.branch.1} $M_{\omega_1}\cap H$ is the direct limit of 
$\langle M_\xi\cap H, \xi\in b\cap \alpha\rangle$, 
\item\label{L.branch.2}  
$\pi^{-1}$
  and $j_{\alpha\omega_1}$  agree on $M_\alpha\cap \bar H$  
 and in particular
 \begin{enumerate}
 \item 
 $\pi^{-1}[M_\alpha\cap \bar H]$ is included in 
$M_{\omega_1}\cap H$, and
\item $\crit(j_{\alpha\omega_1})= \alpha$, 
 \end{enumerate} 
\item \label{L.branch.3} 
If $\gamma$ is the strength of the extender $E_{\alpha'}$ such 
 that $\Ult(M_\alpha,E_{\alpha'})$ is the successor of $M_\alpha$ along $b$ 
 then  $j_{\alpha\omega_1}=j_{E_{\alpha'}}\circ i$ for an elementary embedding $i$
such that  $i\rs \V_\gamma$  is the identity, 
 \item \label{L.branch.4} $M_{\omega_1}\cap \cP(\alpha+1)
 =M_\alpha\cap  \cP(\alpha+1)$.
 \end{enumerate} 
\end{lem} 

\begin{proof} Clause \eqref{L.branch.0} follows by \eqref{IT.4} and \eqref{IT.8+}. 
By elementarity in $H$ it holds that $M_{\omega_1}$ is the direct limit of 
$M_\xi$, for $\xi\in b$, and therefore \eqref{L.branch.1} follows. 
By applying this and \eqref{L.branch.0}, \eqref{L.branch.2} follows as well. 
Let $\gamma$ be as in \eqref{L.branch.3}. 
We have $j_{\alpha\omega_1}=j_{E_{\alpha'}}\circ i$ for some $i$. By Lemma~\ref{L.no-overlaps}
we have $\crit(i)\geq \gamma$ and  \eqref{L.branch.3} follows.  

Clause \eqref{L.branch.4} is a consequence of \eqref{L.branch.2}. 
\end{proof}

\subsection{The iteration game} 
In what follows elements of an iteration tree will be proper classes instead of 
sets. The arguments can be formalized within ZFC by using the standard
reflection and compactness devices. We leave out the well-known details.   
We define a two-player game of transfinite length $\zeta$
in which the players build an iteration tree, starting from a model~$M$ 
and a system of extenders $\vec E$ in $M$. 
Let $M_0=M$. In his $\alpha$-th move \playerI\ picks 
an extender $E_\alpha$ in~${\vec E}^{M_\alpha}$
such that the strength of $E_\alpha$ is greater than the strength of $E_\beta$ for all
 $\beta<\alpha$. 
Then the referee finds the minimal $\beta\leq \alpha$ such 
that  (writing $\crit(E)$ for the critical point of the elementary embedding $j_E$)
$\V_{\crit(E_\alpha)+1}\cap M_\beta=\V_{\crit(E_\alpha)+1}\cap M_\alpha$. 
Hence~$M_\beta$  is the earliest model in the iteration to which $E_\alpha$ can be applied. 
Referee then   defines 
$$
M_{\alpha+1}=\Ult(M_\beta,E_\alpha), 
$$
with $j_{\beta\alpha+1}$ being the corresponding embedding. 
The referee also extends the tree order $T$ by adding $\alpha+1$ as 
an immediate successor to $\beta$. 
At a limit stage $\alpha$ \playerII\ picks a maximal branch $\langle M_\xi\mid \xi\in b\rangle$ 
of $T$ such that~$b$ is cofinal in $\alpha$ and lets $M_\alpha$ be
the direct limit of the system $\langle M_\xi, j_{\xi,\eta}\mid  \xi<\eta\in b\rangle$. 
If $M_\alpha$ is well-founded then we identify it with its transitive collapse. 

The first player who disobeys the rules loses. Assume both players obeyed the rules of the 
iteration game. 
If $M_\alpha$ is ill-founded then the game is over and I wins. 
If all $M_\alpha$ are well-founded, then \playerII\ wins,
and otherwise \playerI\ wins. 

 For definiteness, we call the above game the \markdef{$(\vec E,\zeta)$-iteration game in $M$}. 
We shall suppress $\vec E$ and $M$ whenever they are clear from the context. 
An \markdef{($\vec E,\zeta$)-iteration strategy} 
is a winning strategy for \playerII\ in the iteration game of length~$\zeta$. 
A pair $(M,\vec E)$ is \markdef{($\vec E,\zeta$)-iterable} if \playerII\ has a $\vec E,\zeta$-winning strategy. 
An \markdef{($\vec E,\zeta$)-iteration} of $M$ is an elementary 
embedding $j_{0\zeta}\colon M\to M_\zeta$
extracted from an $\vec E$-iteration tree on $\zeta$. 
A model is \markdef{fully iterable} if it is $(\vec E,\zeta)$-iterable for every ordinal $\zeta$. 

Universally Baire sets of reals were introduced in  \cite{feng92universally}
and are defined in \S\ref{S.Cohen}. 

\begin{thm}[{Martin--Steel, \cite{martin94iterationtrees}}]\lbl{T.MS}
Assume there exist $n$  Woodin cardinals and a measurable above them all. 
For every 
 $a\in \mathbb R$ and every $m\leq n$ there exists an 
  inner model containing $a$ and $m$ Woodin cardinals, 
denoted by $M_m(a)$. It 
 is $(\omega_1+1)$-iterable and its Woodin cardinals 
 are countable ordinals in $V$. 

If there are class many Woodin cardinals then 
every $M_n(a)$ is fully iterable in every forcing extension 
and the iteration strategy is coded by a universally Baire set of reals. 
\qed
\end{thm}

The assumption that there are class many Woodin cardinals is not optimal. 
For the case when $n=1$ see the proof
 of Theorem~\ref{T.M_1} below. 
 
The following lemma will be needed in \S\ref{S.WlimW}. 

\begin{lem}\label{L.uB.strategy} 
Assume $\Sigma$ is an  $\omega_1$-iteration strategy that is universally Baire
and moreover that all sets projective in $\Sigma$ are universally Baire. 
Then~$\Sigma$  can always be 
extended to a  full iteration strategy. 
\end{lem} 

\begin{proof}  Since the statement `$\Sigma$ is a winning iteration strategy' is 
projective in $\Sigma$ it is forcing-absolute. 
One constructs an iteration strategy $\overline\Sigma$ for \playerII\ by induction on limit ordinals
 $\alpha\geq \omega_1$. The recursive hypothesis is that every play of the iteration game
 in which \playerII\ obeyed $\overline\Sigma$
 has  the following property. If $\kappa$ is the cardinality of the resulting 
 iteration tree, then in the extension by Levy collapse 
 $\Coqq\omega{\kappa}$  of $\kappa$ to $\omega$ 
 tree $T$ is the result of a play of an iteration game in which \playerII\ has obeyed~$\Sigma$. 
 
Let $T$ be an iteration tree of height $\alpha$ resulting from an iteration game
in which \playerII\ has  obeyed the extension of $\Sigma$ constructed so far. 
Go to a forcing extension by Levy collapse 
$\Coqq\omega{|\alpha|}$, 
 and let $\bfb$ be the $\alpha$-branch chosen by~$\Sigma$. 

 We claim that for every $\beta<\alpha$ the condition whether $\beta\in \bfb$ 
 is decided by the maximal condition. Otherwise, fix $\beta$ and choose  $p_1$ and $p_2$ 
 such that $p_1\forces \check \beta\in\dot \bfb$ and $p_2\forces \check \beta\notin\dot \bfb$. 
 Let $G_1\subseteq \Coqq\omega{\alpha}$ be a generic filter such that $p_1\in G_1$. 
 By using the homogeneity of $\Coqq\omega{\alpha}$ 
 in $\V[G_1]$ we can choose
a generic filter~$G_2$ such that $p_2\in G_2$ and $\V[G_1]=\V[G_2]$. 
Therefore in $\V[G_1]=\V[G_2]$ we have that $\Sigma$ does not choose a unique 
$\alpha$-branch of $T$, a contradiction. 

Therefore branch $\bfb$ is decided in the ground model. 
We can therefore extend strategy $\overline\Sigma$ by having \playerII\ choose $\bfb$ at the $\alpha$-th stage. Since well-foundedness is absolute for forcing extensions, this defines  a winning
$(\alpha+1)$-iteration strategy. 
\end{proof} 

\section{Genericity iterations}
\label{S.GI} 
In the present section we formulate and prove results that make the extender algebra unique. 

Assuming $M$ is sufficiently iterable, we may talk about iteration strategies for \playerI. 
These are the strategies that, when played against \playerII's winning strategy, produce models 
(necessarily well-founded) with desirable properties.

\begin{thm} \lbl{T1}Assume $(M,\vec E)$ is $(\omega_1+1)$-iterable
and $\vec E$ witnesses a countable ordinal $\delta$ is a Woodin cardinal in $M$. 
Then for  every 
$x\subseteq \omega$ there is a (well-founded) countable iteration $j\colon M\to M^*$ such that 
$x$ is $j(\cWdo(\vec E))$-generic  over~$M^*$. 
\end{thm} 

\begin{proof} 
By Lemma~\ref{L4} and Lemma~\ref{L3} we only 
 need to assure $x\models j(\cTdo(\vec E))$. 
 Define a strategy for \playerI\ for building an $\vec E$-iteration 
tree with $M_0=M$ as follows. 
Assume $\langle T, M_\xi, E_\xi\mid \xi\leq\alpha\rangle$ has been constructed. 
If $x\models j_{0\alpha}(\cTdo(\vec E))$ then $j_{0\alpha}\colon M\to M_\alpha$ is the required iteration and we stop. 
Otherwise, let $\lambda$ be the minimal cardinal 
such that there are $\vec \varphi$,  $\kappa$, 
and a $\vec \varphi,\lambda$-strong extender $E\in j_{0\alpha}(\vec E)$ with $\crit(E)=\kappa$ 
such that $x\not\models \Psi(\vec\varphi,\kappa,\lambda)$. 
 Fix  such $\vec \varphi$, $\lambda$ and $E$. We have 
 $x\not\models \bigvee_{\xi<\kappa} \varphi_\xi$ 
 and $x\models \bigvee_{\xi<\lambda} \varphi_\lambda$. 
 and note that  $\lambda<j_{0\alpha}(\delta)$. 
Then  let \playerI\ play $E_\alpha=E$. 
Note that $j_{E_\alpha}(\kappa)\geq \lambda$. 

This describes the iteration strategy for \playerI. 
We claim that if \playerII\ responds with his winning  strategy
then the process of building the iteration tree terminates at some countable stage. 
Assume otherwise. Let $\langle T, M_\xi, E_\xi\mid \xi<\omega_1\rangle$ be the resulting iteration tree
and let $b\subseteq \omega_1$ be its cofinal branch such that $M_{\omega_1}$, the direct limit of 
$\langle M_\xi\mid \xi\in b\>$, is well-founded.

Fix a  countable $H\prec H_{(2^{\aleph_1})^+}$ containing everything relevant. 
Let $\alpha=H\cap \omega_1$ and let $\bar H$ be the transitive collapse of $H$, 
with $\pi^{-1}\colon \bar H\to H(\theta)$ the inverse of the collapsing map.
Since the iteration did not stop at stage $\alpha$, 
we can consider the extender~$E$ applied  to $M_\alpha$ along the well-founded $\omega_1$-branch of $T$.
Note that $E$ may be different from $E_\alpha$, the extender chosen in $M_\alpha$ during the run of the iteration game. 
Instead, $E$ is equal to $E_\beta$ for some $\beta\geq \alpha$.  

By \eqref{L.branch.3} of Lemma~\ref{L.branch} we have that $\pi^{-1}$ and $j_{\alpha\omega_1}$ agree on 
$M_\alpha\cap \bar H$ because $i$ is the identity on $\V_\gamma$. 
In particular, $\crit(E)=\alpha$. 

By the choice of the iteration, 
$E$ is $\vec\varphi,\lambda$-strong for  some $\vec\varphi$ and $\lambda$ such that 
$\vec\varphi$ reflects to $\alpha$ but $x\not\models\bigvee_{\xi<\alpha} \varphi_\xi$ and
$x\models  \bigvee_{\xi<\lambda}\varphi_\xi$. 
By definition we have $\varphi_\eta\in M_\alpha\cap \V_\alpha$ for all $\eta<\alpha$, 
hence the formula $\bigvee_{\eta<\alpha} \varphi_\eta$ belongs to~$\V_{\alpha+1}$. 
By \eqref{L.branch.0} of Lemma~\ref{L.branch} there are $\xi<_T \alpha$ and $\vec\psi \in M_\xi$
such that $\vec\varphi=j_{\xi\alpha}(\vec \psi)$. 

Since $j_{\alpha\omega_1}$ and $j_E$ agree
on $M_\alpha\cap \V_\lambda$, we have
that  $j_{\alpha\omega_1}(\vec\varphi)$ implies $\bigvee_{\xi<\lambda} \varphi_\xi$, 
hence $x\models j_{\alpha\omega_1}(\vec \psi)$ holds in~$H$. 
Thus for some $\xi\in H\cap \omega_1$ we have $x\models \varphi_\xi$
and therefore $x\models \bigvee_{\xi<\alpha} \varphi_\xi$, a contradiction. 
\end{proof} 



A number of  extensions of Theorem~\ref{T1} in different directions are in order (some of their obvious 
common generalization are being omitted).

 Assume an extender $E$ and a forcing  $\bbP$   
are such that for some $\kappa$ we have $\crit(j_E)>\kappa$  
and $\bbP$ is  in $\V_\kappa$. 
Then $E$ still defines an extender 
 in the extension by $\bbP$ and this extender is 
 $\lambda$-strong for every $\lambda$ for which $E$ is 
$\lambda$-strong. This is essentially a consequence of the Levy--Solovay result 
that a measurable cardinal cannot be destroyed by a small forcing (\cite{Kana:Book}). 
The converse is also true (\cite{HaWo:Small}): 
in a forcing extension $\V[G]$ by forcing $\bbP\in \V_\kappa$ 
every $\lambda$-strong elementary embedding of $\V[G]$ into an inner model 
with critical point $\kappa$ is a lift of a  ground model $\lambda$-strong elementary embedding with critical point $\kappa$. This is a bit deeper than the corresponding result for elementary embeddings that are not necessarily $\lambda$-strong.  

In the following lemma and elsewhere we shall slightly abuse the notation and 
denote the  extender obtained from $E$ in a forcing extension by the same letter.

\begin{thm}  \lbl{T1.1.1}Assume $(M,\vec E)$ is $(\omega_1+1)$-iterable
and $\vec E$ witnesses a countable ordinal $\delta$ is a Woodin cardinal in $M$. 
Assume moreover $\kappa<\min\set{\crit(j_E)}{ E\in \vec E}$ and $\bbP\in \V_\kappa\cap M$ 
is a forcing notion. 
If $G\subseteq \bbP$ in $V$ is $M$-generic then 
for every $x\subseteq \omega$ there is a (well-founded) 
countable iteration $j\colon M[G]\to M^*[G]$ such that 
$x$ is $j(\cWdo(\vec E))$-generic  over~$M^*[G]$. 
\end{thm}

\begin{proof} By the above discussion $\delta$ is still a Woodin 
cardinal in $M[G]$ as witnessed by 
 $\vec E$.  
The proof is similar to the proof of Theorem~\ref{T1}. 
In his strategy,   player~I computes  
$\Psi(\vec\varphi,\kappa,\lambda)$ in $M[G]$ instead of $M$. 
Since $G\in V$, this describes an iteration strategy of \playerI\ in $V$. Thus \playerII\ can respond to it by using his 
winning iteration strategy and 
the other details of the proof are identical. 

As a matter of fact, a simple proof shows that the full iterability of $M$ implies 
the full iterability of $M[G]$ (essentially using the same strategy). 
The point is that, since $\bbP$ is small, 
 the ultrapowers of $M$ lift to ultrapowers of $M[G]$. 
 Therefore 
this is a consequence of Theorem~\ref{T1}. 
\end{proof}

\begin{proof}[The second proof of Lemma~\ref{L.universal}]
We now provide a different proof that 
if $\gamma$ is any cardinal less than the least critical point of each extender in $\vec E$
and $\bbP$ is a forcing notion of cardinality $\gamma$ 
then in  $\cWdg(\vec E)$ there is a condition $p$ that forces that 
$\bbP$ is a regular subordering of $\cWdg(\vec E)$.  
The present proof will  also show that 
there exist  
  condition $q\in \cWdg(\vec E)$ and condition $r$ in $\bbP$ that 
force that  $\bbP$ is forcing-equivalent to $\cWdg(\vec E)$.
Assume $G\subseteq \bbP$ is generic over $M$. 
We may assume $\bbP=(\gamma,<_{\bbP})$ for  some 
ordering~$<_{\bbP}$ on $\gamma$. Since $\delta$ is a countable ordinal and $\gamma<\delta$, 
by using Theorem~\ref{T1} we can find an iteration $j\colon M\to M^*$ such that 
$G$ is $j(\cWdg(\vec E))$ generic over $M^*$. By our assumption $\gamma$ is
smaller than the critical point of $j$ and therefore $j(\bbP)=\bbP$. 
Hence in $M^*$ there is $q_0$ in $j(\cWdg(\vec E))$ and $r\in \bbP$ that forces
$\bbP$ and $j(\cWdg(\vec E))$ are forcing equivalent. By the elementarity, 
this is true in $M$ for $\bbP$ and $\cWdg(\vec E)$. 
\end{proof}

Here is yet another variation of genericity iterations, also due to Woodin. 
Its proof  was  sketched in the appendix to  \cite{NeZa:ProperA} (not included in the published version  
\cite{NeZa:Proper}).

\begin{thm} \label{T.NZ} Assume $\bbP$ is a forcing notion of cardinality $\kappa$ and $\dot x$ is 
a $\bbP$-name for a real. Assume $(M,\vec E)$ is $(\kappa^++1)$-iterable
and $\vec E$ witnesses a countable ordinal $\delta$ is a Woodin cardinal in $M$. 
Then there is a (well-founded) iteration $j\colon M\to M^*$ of length $<\kappa^+$ such that 
$\bbP$ forces $\dot x$ is $j(\cWdo(\vec E))$-generic  over~$M^*$. 
\end{thm} 

\begin{proof}[Proof of Theorem~\ref{T.NZ}]  has similar structure as the proof of Theorem~\ref{T1}. 
More precisely, \playerI\ and \playerII\ play the iteration game in which \playerII\ follows his winning strategy while \playerI\ at each move chooses a bad extender with the least possible strength. In the present situation an extender is bad if some condition in $\bbP$ forces that $\dot x$ violates an axiom associated with this extender. 
The game stops when $\bbP$ forces that $\dot x$ satisfies the theory corresponding to $\vec E$ and therefore 
that it is $j(\cWdo(\vec E))$-generic over the iterate $M^*$.

Showing that the construction  terminates before the $\kappa^+$-th stage requires
taking the elementary submodel $H$ of cardinality $\kappa$ and an appropriate 
analogue of Lemma~\ref{L.branch}. 
\end{proof} 

In \cite{NeZa:Proper} it was proved that 
if the forcing~$\bbP$ is proper then one can find a countable genericity iteration. The authors of
\cite{NeZa:Proper} 
also pointed out that this is not necessarily true when $\bbP$ is only assumed to be semiproper.

\subsection{Genericity iterations for subsets of $\omega_1$}

We finally turn to applications of the algebra $\Wdd$ with $\delta$ generators.

\begin{thm} \lbl{T1.1}Assume $(M,\vec E)$ is $(\omega_1+1)$-iterable
and $\vec E$ witnesses a countable ordinal $\delta$ is a Woodin cardinal in $M$. 
Then for every $x\subseteq \omega_1$ there is a (well-founded) 
countable iteration $j\colon M\to M^*$ such that 
$x\cap j(\delta)$ is $j(\cWdd(\vec E))$-generic  over~$M^*$. 

Furthermore, we can assure $j(\delta)\in \bfC$ for any club $\bfC\subseteq \omega_1$ 
given in advance. 
\end{thm} 

\begin{proof}  
The strategy for \playerI\ is   identical to the strategy used in the 
 proof of Theorem~\ref{T1}, and the proof of the latter shows that the iteration 
 terminates at some countable stage at which 
 $x\cap j(\delta)\models j(\cTdd(\vec E))$. 

Now we show how to  assure $j(\delta)\in \bfC$. 
Let $c_\alpha\subseteq \omega$ be a real coding the $\alpha$-th element of $\bfC$
and let $\overline \bfC$ be a subset of $\omega_1$ such that its intersection with the interval 
$[\omega\alpha,\omega(\alpha+1))$ codes
$c_\alpha$. Identify  $x$ with a set of even ordinals and $\overline\bfC$ with a set of odd ordinals
and let $y$ be the result.  
By applying the first part of this theorem to $y$ find an iteration $j\colon M\to M^*$ such that 
$x\cap j(\delta)$ and $\overline\bfC\cap j(\delta)$ are both $j(\cWdd(\vec E))$-generic over $M^*$. 
Let $G\subseteq [j(\cWdd(\vec E))$ be the generic filter defined by $x\cap j(\delta)$ and $\overline\bfC\cap j(\delta)$. 
Since $c_\xi\in M^*[G]$ for all $\xi<j(\delta)$, we have that in $M^*[G]$ 
all ordinals below $j(\delta)$-th element of $\bfC$ are collapsed
to $\aleph_0$.  
By the $j(\delta)$-cc of $j(\cWdd(\vec E))$, $j(\delta)=\aleph_1^{M^*[G]}$. 
Assume $\bfC\cap j(\delta)$ was bounded and pick $\xi$ such that $\sup\bfC\cap j(\delta) <\xi<j(\delta)$. 
Then $c_\xi$ codes an element of $\bfC$ greater than $j(\delta)$. But this means that $j(\delta)$ is collapsed
to $\aleph_0$ in $M^*[G]$, contradicting the above.   
\end{proof}

The following theorem is an analogue of Theorem~\ref{T1} for an arbitrary set of ordinals
and it will be used in the proof of Theorem~\ref{T.M_1}. 
    More useful versions 
    (Theorem~\ref{T.WlimitW} and Theorem~\ref{T1.delta}) 
    require  stronger large cardinal assumptions.

\begin{thm} \label{T1+} Assume $(M,\vec E)$ is  fully iterable
and $\vec E$ witnesses a countable ordinal $\delta$ is a Woodin cardinal in $M$. 
Then for  every set of ordinals 
$x$ there is a (well-founded)  iteration $j\colon M\to M^*$ of length $<|x|^+$ such that 
$x$ is $j(\cWdd(\vec E))$-generic  over~$M^*$. 
\end{thm} 

\begin{proof} Let $E_0\in \vec E$ be an extender with minimal strength. 
Iterate~$E_0$ to obtain an iteration $j_0\colon M\to M_0$ such that $x\subseteq j(\delta)$. 
Since this is a linear iteration, it is well-founded. By the minimality, the strength of every 
extender in $j(\vec E)$ is greater than the strength of all iterates of $E_0$ used 
in~$j_0$. Therefore for every iteration $i\colon M_0\to M^*$ of $(M_0,j_0(\vec E))$ 
we have that $i\circ j_0\colon M\to M^*$ is an iteration of $(M,\vec E)$, and therefore 
$(M_0,j_0(\vec E))$ is fully iterable. 

From this point on the proof is analogous to the proof of Theorem~\ref{T1}. 
Define a strategy for \playerI\ for building an iteration 
tree starting with 
 $M_0$ as follows. Assume $\langle T, M_\xi, E_\xi\mid \xi\leq\alpha\rangle$ has been 
 constructed. 
If $x\models (j_{0\alpha}\circ j_0)(\cTdd(\vec E))$ then $j_{0\alpha}\circ j_0
\colon M\to M_\alpha$ is the required iteration and we stop. 
Otherwise, let $\lambda$ be the minimal cardinal 
such that there are $\vec \varphi$,  $\kappa$, 
and a $\vec \varphi,\lambda$-strong extender $E\in (j_{0\alpha}\circ j)(\vec E)$ 
with $\crit(E)=\kappa$ 
such that $x\not\models \Psi(\vec\varphi,\kappa,\lambda)$. 
 Fix  such $\vec \varphi$, $\lambda$ and $E$. We have 
 $x\not\models \bigvee_{\xi<\kappa} \varphi_\xi$ 
 and $x\models \bigvee_{\xi<\lambda} \varphi_\lambda$. 
 and note that  $\lambda<(j_{0\alpha}\circ j_0)(\delta)$. 
Then  let \playerI\ play $E_\alpha=E$. 
Note that $(j_{E_\alpha}\circ j_0)(\kappa)\geq \lambda$. 

This describes the iteration strategy for \playerI. A proof that if 
 \playerII\ responds with his winning  strategy
then the process of building the iteration tree terminates at some stage before $|x|^+$ 
is identical to the proof of Theorem~\ref{T1}, using the variant  of Lemma~\ref{L.branch}
for an iteration tree of height $|x|^+$.
\end{proof}

The iterability assumption of Theorem~\ref{T.WlimitW}  below is almost the  
 strongest known result of its 
kind. Its consistency modulo large cardinals was
 proved by Neeman \cite{Nee:Inner}. 
 If $\CH$ holds and $\set{x_\xi}{ \xi<\omega_1}$ is an enumeration of
all reals  we write 
\[
\bbR\rs \gamma=\set{x_\xi}{ \xi<\gamma}. 
 \]
 The
extender algebra with $\delta$ generators 
need not satisfy \eqref{I.WlimW.3}
of  the following theorem,  even when it collapses $\delta$ to $\aleph_1$. 

\begin{thm}\label{T.WlimitW}
Assume $(M,\vec E)$ is  $(\omega_1+1)$-iterable
and $\vec E$ witnesses $\delta$ is a Woodin limit of Woodin cardinals in $M$.   
Furthermore assume $\CH$ holds  and $A$ is a set of reals. 
Then there 
are a (well-founded)  iteration $j\colon M\to M^*$ of countable length  and  
 an $M^*$-generic filter $G\subseteq j(\cWdd(\vec E_0))$ such that the following hold:
\begin{enumerate}
\item  \label{I.WlimW.1}  $\bbR^{M^*[G]}=\bbR\rs j(\delta)$,
\item \label{I.WlimW.2}$A\cap (\bbR\rs j(\delta))\in M^*[G]$, and 
\item \label{I.WlimW.3}Every real in $M^*[G]$ is $M^*$-generic  for a forcing of cardinality $<j(\delta)$. 
\pushcounter
\end{enumerate}

\begin{enumerate}
\popcounter
\item\label{I.WlimW.4}  There is $B\subseteq \omega_1$ such that  $N=\LL[M,B\rs\alpha]$ satisfies
 $\bbR^{M^*[G]}=\bbR^N$ 
\pushcounter
\end{enumerate}

\end{thm} 

\begin{proof} 
Let $\Sigma$ be the winning strategy for \playerII\ in the iterarion game for $M,\vec E$. 
It is naturally identified with a set of reals and therefore, using~$\CH$, 
 with a subset of $\omega_1$. 
Fix a subset $B$ of $\omega_1$ so that its intersection with $\omega$ codes $M$
 and $B$ codes $\bbR$, $A$ and $\Sigma$.
 Let $\bfC$ be the club of all $\alpha<\omega_1$ such 
$B\rs \alpha$ is an elementary submodel of $B$. 

Applying Theorem~\ref{T1.1} find an iteration $j\colon M\to M^*$ such that 
(with $\alpha=j(\delta)$) we have that 
$B\cap \alpha$
is $j(\cWdd(\vec E))$-generic over $M^*$ and $\alpha$ is a limit point of club $\bfC$. 
Therefore for cofinally many $\xi<\alpha$ we have $B\rs \xi\prec B$. 
 Let $G$ be the generic filter defined by $B\rs\alpha$. 
Note that \eqref{I.WlimW.2} is immediate.
We shall now prove \eqref{I.WlimW.4}. Note that the converse inclusion is immediate.

Only the direct inclusion requires a proof. Let 
$\langle T, M_\eta, E_\xi\mid \eta\leq\alpha, \xi<\alpha\rangle$
 be the iteration tree constructed during 
the course of finding~$M^*$. Since the iteration strategy of \playerI\ is definable and $\Sigma\rs\alpha$
is an elementary submodel of $\Sigma$, this tree (except possibly its final $\alpha$-branch) can be constructed in $N$.

Now assume $\dot y$ is a nice $j(\cWdd(\vec E))$-name for a real. By $\alpha$-cc there
is a Woodin cardinal $\gamma<\alpha$ in $M^*$ such that all conditions in $\dot y$ use only generators 
less than $\gamma$.     
Let  $\vec E(\gamma)$
be the extenders in $j(\vec E)$ whose strength is below $\gamma$. Then $\vec E(\gamma)$ witnesses Woodinness
of $\gamma$ in $M^*$. Also, $G\rs\gamma$ satisfies all axioms of theory $\cT(\vec E(\gamma))$ in $M^*$. 
Let $\beta<\alpha$ be large enough so that $M_\beta \cap \V_{\gamma+1}=M^*\cap \V_{\gamma+1}$ and
therefore all axioms of $\cT(\vec E(\gamma))$ are satisfied in $M_\beta$.

 Therefore  we have proved 
 \begin{enumerate}
 \popcounter
 \item \label{I.WlimW.5} $G\rs \gamma$ is 
generic for $\bbP=\cW_{\gamma,\gamma}(\vec E_\gamma)$ over both $M_\beta$ and $M^*$.
\pushcounter
\end{enumerate} 
Since all conditions occurring in $\dot y$ belong to $\bbP$, 
 the interpretation of $\dot y$ with respect to the filter in $\bbP$ defined by $G\rs \gamma$ over $M_\beta$  
is definable from $M$ and $G\rs \alpha$ and therefore belongs to $N$. 
Since this interpretation coincides with the interpretation of $\dot y$
with respect to the filter in $j(\cWdd(\vec E))$ defined by~$G$, and since $\dot y$ was an arbitrary name for a real, 
the direct inclusion of \eqref{I.WlimW.4} follows. 

Since $N$ is a definable inner model in $M^*[G]$, \eqref{I.WlimW.1} follows.

It remains to prove~\eqref{I.WlimW.3}. 
 The elementarity 
of $\bbR\rs \alpha$ in $\bbR$ implies that  
 $j(\cWdd(\vec E))$  collapses all cardinals below $\alpha$ to $\aleph_0$. 
By \eqref{I.WlimW.5}   each  $x_\xi$, for $\xi<\gamma$, is generic over $M^*$ 
for a small forcing. Since $\delta$ is a limit of Woodins all reals in~$\bbR\rs \alpha$ are 
generic over $M^*$ for a small forcing notion. 
 \end{proof}

The assumption of the existence of an $(\omega_1+1)$-iterable model with a measurable 
Woodin cardinal used in the following theorem  is presently beyond reach of the  inner model theory. 

\begin{thm} \lbl{T1.delta}Assume $(M,\vec E)$ is $(\omega_1+1)$-iterable and $\vec E$ witnesses
a countable ordinal $\delta$ is a measurable 
Woodin cardinal in $M$. 
Then for every $x\subseteq \omega_1$ there is a (well-founded) $\omega_1$-iteration $j\colon M\to M^*$ such that 
$x$ is $j(\Wdd(\vec E))$-generic  over~$M^*$. 
\end{thm} 

\begin{proof} This  proof is an extension of the proof of Theorem~\ref{T1.1}. 
We need to assure $x\models j(\cTdd(\vec E))$ for an $\omega_1$-iteration $j$. 
Define a strategy of \playerI\ for building an iteration 
tree starting with $M_0=M$ as follows. Assume $\langle T, 
M_\eta, E_\xi\mid \eta\leq\alpha,\xi<\alpha\rangle$ has been constructed. 

(a) Assume there is a sequence $\vec \varphi$ in $M_\alpha$ that reflects to some $\kappa$ but 
$x\not\models\Psi(\vec\varphi,\kappa,\lambda)$ for some $\lambda$ satisfying 
$\lambda<j_{0\alpha}(\delta)$  and
 $\lambda>\sup_{\xi<\alpha}(\lambda_{E_\xi})$. 
Choose the minimal $\lambda$ with this property, fix the 
appropriate~$\vec \varphi$ and $\kappa$ so 
that $x\not\models\Psi(\vec\varphi,\kappa,\lambda)$.
Then let~$E_\alpha$ be a $(\vec\varphi,\lambda)$-strong  extender in $M_\alpha$ such that $\crit(E_\alpha)=\kappa$. 
Note that $j_{E_\alpha}(\kappa)\geq \lambda$.

(b) Now assume $\alpha$ is countable  and $x\models j_{0\alpha}(\cTdd(\vec E))$. 
Since $\delta$ is a measurable Woodin cardinal, use  normal measure $U$ on 
$j_{0\alpha}(\delta)$ to define $M_{\alpha+1}=\Ult(M_\alpha,U)$ and 
let $j_{\alpha,\alpha+1}:M_\alpha\to M_{\alpha+1}$.

This describes the iteration strategy for \playerI. Now consider the iteration tree formed when \playerII\ plays
his winning strategy against the iteration strategy just defined for player I. 
 By  Theorem~\ref{T1.1}, the set $C$ of all stages $\alpha$ in the iteration  
  such that   $x\models j_{0\alpha}(\cTdd(\vec E))$ and  
 $\alpha=j_{0\alpha}(\delta)$ is unbounded, and it is therefore a club.  
Note that   $x\cap \alpha$ is $j_{0\alpha}(\Wdd(\vec E))$-generic for all $\alpha\in C$. 

By the choice of extenders and Lemma~\ref{L.no-overlaps},  for $\alpha\in C$ the critical points 
of embeddings constructed after $\alpha$th stage will never drop
below $j_{0\alpha}(\delta)$. 

The critical sequence defines a club in $\omega_1$ 
and  an $\omega_1$-iteration $\langle N_\xi, j_{\xi\eta}\mid \xi\leq \eta\leq \omega_1\rangle$
such that $N_0=M$, for $\xi<\eta$ the embedding 
$j_{\xi\eta}\colon N_\xi\to N_\eta$ has~$\alpha_\xi$ as its critical point, 
and $x\cap \alpha_\xi\models j_{0\xi}(\cTdd(\vec E))$ for 
all $\xi$. Since the critical points are increasing, this implies that $x\models j_{0\omega_1}(\cTdd(\vec E))$, 
as required. 
\end{proof}

 \section{Absoluteness} By `$M_1$ exists' 
 we denote the statement `There exists an inner model 
of the form $\LL[\vec E]$  with a Woodin cardinal whose Woodinness is witnessed by extenders in $\vec E$.'
Similarly, if $a$ is a real then `$M_1(a)$ exists'
denotes the statement
`There exists an inner model 
of the form $\LL[\vec E]$ with a Woodin cardinal  whose Woodinness is witnessed by extenders in $\vec E$ and which contains $a$.'
Variants such as `$M_n$ exists' or `$M_1$ exists and it is fully iterable' are interpreted similarly. 
The following lemma and its  variations  will be used tacitly in order to furnish the assumptions of genericity iteration theorems. 

\begin{lem} Assume $M_1$ exists and it is fully iterable. Then there exists a fully iterable inner  
model  with a Woodin cardinal $\delta$ such that $\delta$ is a countable ordinal in $V$. 
\end{lem} 

\begin{proof} Take a countable elementary submodel $N$ of $M_1$. Theorem~\ref{T.M_1}
 implies that $N$ has a sharp. We can therefore find an iteration of $N$ of length Ord  whose resulting model is an 
inner model with a Woodin cardinal which is a countable ordinal in $V$. The iteration strategy for $N$ 
is obtained by copying the iteration strategy of $M_1$ (see \cite{Sar:Short}). 
\end{proof}

 \label{S.Ab}
 \subsection{Absoluteness in $\LL(\bbR)$}
The following result and its proof are a prototype for the main result of this note, Theorem~\ref{T.0}. 

\begin{thm}\label{T.00} 
 Assume  $M_1(a)$ is fully iterable in all forcing extensions for all $a\in \bbR$. 
 Then all $ \bfSigma^1_3$ statements are forcing absolute. 
\end{thm} 

\begin{proof} 
A $\bfSigma^1_3$ statement $\varphi(a)$ with parameter $a\in \bbR$ 
is of the form 
$(\exists x)\psi(x,a)$ where $\psi$ is  ${\Pi^1_2}$. 
To $\varphi$ we associate 
a sentence $\varphi^*$ of $M_1(a)$ (see Theorem~\ref{T.MS}) 
stating that there exists a forcing notion $\bbP$ forcing $\varphi$. 
We claim that $\varphi$ holds (in $V$) if and only if  $\varphi^*$ holds in $M_1(a)$. 
This will suffice since the iterability  of $M_1(a)$ is 
not changed by forcing. 
If $\varphi^*$ holds in $M_1(a)$ then since $M_1(a)\cap \V_{\delta+1}$ 
is countable we can find an $M_1(a)$-generic filter $G\subseteq \bbP$ for $\bbP$ referred to in $\varphi^*$. 
If $M_1(a)[G]\models \psi(x,a)$ then by Shoenfield's 
absoluteness theorem $\psi(x,a)$ holds in $V$. 
For the converse implication, assume $\psi(x,a)$ holds in some forcing extension 
of $V$ for some $x\in \mathbb R$. Since  $M_1(a)$ is fully iterable in all forcing extensions, 
apply 
 Theorem~\ref{T1} to find   a countable  iteration 
  $j\colon M_1(a)\to M^*$ such that~$x$ is generic 
over~$M^*$. Then (again using Shoenfield) $M^*\models \varphi^*$, and by elementarity $M_1(a)\models \varphi^*$. 
\end{proof}

The assumptions of Theorem~\ref{T.00} are far from optimal. By a result 
of Martin and Solovay (\cite{martin69basistheorem}), if $\kappa$ is a measurable cardinal
then $\bfSigma^1_3$ sentences are forcing absolute 
for forcing notions in $\V_\kappa$. As a matter of fact, all $\bfSigma^1_3$ sentences 
are absolute between all forcing extensions of $V$ if and only if all sets have sharps 
(Martin--Steel, Woodin; see \eg, \cite{steel06innermodel} for terminology). 
One important fact about Theorem~\ref{T.00} and
its extension, Theorem~\ref{T.00n} below, is that
its proof  is susceptible to far-reaching generalizations. 
Also,   Corollary~\ref{C.n} below is worth mentioning. In the language of 
Woodin's $\Omega$-logic (see \cite[\S 10]{Wo:Pmax2}, \cite{Woo:Omega}, \cite{BaCaLa},  
or \cite{Lar:Three}), it states that if the function $x\mapsto M_1(x)^\#$ is universally Baire then it 
is an $\Omega$-proof for all   true $\bfSigma^1_3$ statements (see \S\ref{S.Oinfty}). 
 As common in inner model theory, by $M_1(x)^\#$ we denote an iterable transitive model of 
ZFC that contains $x$ and has a top measure  (for more on mice see \cite{Schi:ABC}). 

\begin{cor} \label{C.n} Assume $N$ is a transitive model of a large enough fragment of ZFC
that is closed under the map 
$x\mapsto M_1(x)^\#$. Then $N$ is $\bfSigma^1_3$-correct. 
\end{cor} 

\begin{proof} For every real $x$ in $N$,  $N$ can reconstruct $M_1(x)$ and its iteration strategy. 
Therefore the proof of Theorem~\ref{T.00} applies inside $N$. 
\end{proof}


It should be noted that  $M_1$ is not necessarily $\bfSigma^1_3$ correct. 
For example, the fine-structural version of $M_1$ has a $\Delta^1_3$ well-ordering of the reals. 
Similarly, being closed under the sharps (\ie, under the map $x\mapsto x^\#$) 
does not guarantee $\bfSigma^1_3$ correctness since $\LL[\mu]$ has a $\Delta^1_3$ well-ordering of the reals. 
Nevertheless, $M_2$ is $\bfSigma^1_4$ correct   (see \cite{steel95projectivelywellordered}). 
Here we shall only show that a proof analogous to the above gives a weaker result.

\begin{thm} \label{T.00n} Assume $M_n(a)$ is fully iterable in all forcing extensions 
for every $a\in \bbR$. 
Then all $\bfSigma^1_{n+2}$ statements are forcing absolute. 
\end{thm} 

\begin{proof} 
We first prove the case $n=2$. 
A $\bfSigma^1_{4}$ sentence $\varphi$
has the form $(\exists x)(\forall y) \psi(x,y,a)$ for some real parameter 
$a$ and a $\Pi^1_2$ formula $\psi$. 
Let $M_2(a)$ be the minimal model for two Woodin cardinals containing~$a$  
fully iterable in all forcing extensions (Theorem~\ref{T.MS}), 
and let $\delta_0<\delta_1$ be its Woodin cardinals.

To $\varphi$ associate a sentence $\varphi^*$ stating that there is a forcing $\bbP$ in $\V_{\delta_0+1}$
and a $\bbP$-name $\dot x$ for a real 
such that for every forcing $\dot \bbQ\in \V_{\delta_1+1}$ and a $\dot\bbQ$-name~$\dot y$ for a real 
we have 
\[
\forces_{\bbP*\dot\bbQ} \psi(\dot x,\dot y, a). 
\]
We claim that $\varphi$ holds in $V$ if and only if $\varphi^*$ holds in  $M_2(a)$. 

Assume $M_2(a)\models \varphi^*$. Find an $M_2(a)$-generic 
$G\subseteq \bbP$ 
and let $x=\Int_G(\dot x)$. Let $y$ be any real. Let $\vec E_1$ be 
the system of extenders witnessing $\delta_1$ is Woodin in $M_2(a)$. We may 
assume $\min\set{\crit(j_E)}{ E\in \vec E_1}>\delta_0$.  
Using the full iterability of $M_2(a)$ and Theorem~\ref{T1.1.1}  
find a well-founded iteration $j\colon M_2(a)[G]\to M^*[G]$ such that $y$ is $j(\cWdo(\vec E_1))$-generic
over~$M^*[G]$. 
Then by elementarity we have 
$M^*[G][y]\models \psi(x,y,a)$ and by Shoenfield's absoluteness theorem 
 $\psi(x,y,a)$ holds. Since $y$ was arbitrary, we have proved $\varphi$ holds in $V$. 

Now assume $\varphi$ holds in $V$ and let $x\in \cP(\omega)$ be such that 
$(\forall y)\psi(x,y,a)$. By Theorem~\ref{T1} we can  
find an iteration $j_0\colon M_2(a)\to M^*$ such that 
$x$ is $j_0(\cWdo(\vec E_0))$-generic over $M^*$. Fix  $y\in \cP(\omega)$. 
By Theorem~\ref{T1.1.1} there is an iteration $j_1\colon M^*[x]\to M^{**}[x]$ such that 
$y$ is $(j_1\circ j_0)(\cWdo(\vec E_1))$-generic over $M^{**}[x]$. 
Then $\V\models \varphi(x,y,a)$ and by the Shoenfield's 
absoluteness theorem $M^{**}[x]\models \varphi(x,y,a)$. Since $y$ was arbitrary, this shows $\varphi^*$ 
holds in~$M_2(a)$. 

This concludes  proof of the theorem in the case when we have two Woodin cardinals. 
In general case, to a $\bfSigma^1_{n+2}$ 
formula $(\exists x_1)(\forall x_2)\dots \psi(x_1,\dots, x_n,a)$ with $\psi$ being~$\bfSigma^1_2$ one associates a 
formula $\varphi^*$ of $M_n(a)$ stating 
\[
(\exists \bbP_1\in \V_{\delta_0+1})(\exists \dot x_1)
(\forall \bbP_2\in \V_{\delta_1+1})(\forall \dot x_2)\dots
\forces_{\bbP_1*\dot \bbP_2* \dots *\dot\bbP_n}\psi(\dot x_1,\dots, x_n,a)
\]
and proves that $\V\models \varphi$ is equivalent to $M_n(a)\models \varphi^*$ as above. 
\end{proof}


\begin{rem} \label{R.1} 
The following was pointed out  by Menachem Magidor. 
In Theorem~\ref{T.00n}
it is not sufficient to assume that there are $n$ Woodin cardinals and a measurable above. 
Such an assumption can hold if $V=\LL[\vec E]$ for a system of extenders $\vec E$. 
However a forcing extension of $\LL[\vec E]$ in which a large enough cardinal is collapsed to $\aleph_0$ 
is of the form $\LL[x]$, for a real~$x$. Such an extension satisfies the projective statement
`there exist a real~$x$ and a $\Delta^1_2(x)$ well-ordering of $\bbR$.' 

However, the existence of a proper class of Woodin cardinals gives such an absoluteness, and more (see \cite{Lar:Stationary}). 
\end{rem}

By Remark~\ref{R.1} the existence of $M_n(a)$ is a strictly weaker assumption than its  iterability 
in all forcing extensions, needed in Theorem~\ref{T.00}. 
 We write~$M_n$ for $M_n(0)$. Recall that a set of ordinals $X$ \markdef{has a sharp} 
 if there is a nontrivial elementary embedding of $\LL[X]$ into itself. 

The following result is also due to Woodin. 

\begin{thm} \label{T.M_1} The following are equivalent. 
\begin{enumerate}
\item $M_1$ exists and it is fully iterable in all forcing extensions. 
\item Every set has a sharp and there is a proper class model with a Woodin cardinal. 
\end{enumerate}
\end{thm}

\begin{proof} (2) implies (1) is a difficult result using Mitchell--Steel constructions and  
we only prove the implication from (1) to (2). 
 We only need to 
 show that the full iterability of $M_1$ implies that every set has a sharp. 
Assume the contrary, and let $X$ be a set without a sharp. 
Then the Covering Lemma holds in $\LL[X]$, hence the successor $\lambda^+$ of some singular
cardinal~$\lambda$ such that $X\in \V_\lambda$ 
is correctly computed in~$\LL[X]$ (\cite{Kana:Book}). 
By Theorem~\ref{T1+} there is an iteration $j\colon M_1\to M^*$ such that 
$X$ is generic over~$M^*$ and $j(\delta)<\lambda^+$. 
Then $M^*[X]$ correctly computes $\lambda^+$, since it includes~$\LL[X]$. 
On the other hand, by the chain condition  of the extender algebra
$j(\delta)$ remains a cardinal in~$M^*[X]$. A contradiction. 
\end{proof} 

I learned the above proof that  (1) implies (2) from Ralf Schindler. 
This proof also shows that if $M_{n+1}$ is fully iterable then for every set $X$ there
is an inner model including $\LL[X]$ with $n$ Woodin cardinals that has a sharp.

\subsection{Absoluteness for $H(\aleph_2)$}
 Theorem~\ref{T.0} below was proved independently by Steel and Woodin 
and a proof of its strengthening due to  Neeman  can be found in \cite{Nee:Determinacy}. 
Theorem~\ref{T.0} implies that the existence of a model with a measurable Woodin cardinal
that  is fully iterable in all forcing extensions 
would provide  another proof of Woodin's $\Sigma^2_1$-absoluteness theorem (\cite{Wo:Sigma-2-1}; see also 
\cite{Doe:Stationary}, \cite{Lar:Stationary}, or 
\cite{Fa:A_proof}). 
The proof of this theorem will use the following standard forcing fact. 
The assumption that $\V_{\delta}\cap M$ is countable is used only to assure the existence
of generic objects. 

\begin{lem}\label{L.Forcing.1}  Assume $M$ is a 
transitive model of a large enough fragment of ZFC such that $M\cap \V_{\delta}$ is countable. 
Assume  $\bbP$ and $\bbQ$ belong to $M\cap \V_\delta$ and 
$\bbP$ is a regular subalgebra of $\bbQ$. If $G\subseteq\bbP$ is $M$-generic, 
then there is an $M$-generic $H\subseteq \bbQ$ such that $G\in M[H]$. \qed
\end{lem}

\begin{lem}\label{L.reg} 
Assume  $\bbP$ is a forcing 
notion in $M$ with $\delta$-cc of cardinality~$\delta$. If $j\colon M\to M^*$ is 
an elementary embedding with $M^*$ a definable class in~$M$ and
$\crit(j)=\delta$, then $\bbP$ is a regular subordering
of $j(\bbP)$ in $M$. 
\end{lem} 

\begin{proof} We may assume $\bbP\subseteq \V_\delta$. 
Let $\cA$ be a maximal antichain in $\bbP$. By the $\delta$-cc we have $\cA\in \V_\delta$, 
and therefore $j(\cA)=\cA$. 
By the elementarity,~$\cA$ is a maximal antichain in $j(\bbP)$ in $M^*$. 
Since being a maximal antichain is absolute, $\cA$ is a maximal antichain of $j(\bbP)$
in $M$. 
\end{proof} 

%
%

\begin{thm} \label{T.0} 
Assume there exists a model $\mwM$
with a countable ordinal~$\delta$ that is a measurable  Woodin cardinal in $\mwM$
which is fully iterable in all forcing extensions.  
Then to every $\Sigma^2_1$ statement $\varphi$ we can associate a statement~$\varphi^*$
such that if $\V\models \varphi$ then $\mwM\models \varphi^*$
and if $\mwM\models \varphi^*$ and $\CH$ holds then~$\V\models \varphi$. 
\end{thm} 

\begin{proof} The sentence $\varphi$ is of the form $(\exists X\subseteq \bbR)\psi(X)$ where $\psi$ is a statement of~$(H(\aleph_1), X, \in)$. 
To it we associate 
 $\varphi^*$ stating that  some condition in~$\Wdd$ forces $\varphi$ and $|\check \delta|=\aleph_1$.

In order to prove that $\varphi$ implies $\mwM\models \varphi^*$, 
assume $X\subseteq \bbR$ is such that~$\psi(X)$ holds. 
Go to a forcing extension of $V$  with the same reals that satisfies $\CH$. 
Fix $Y\subseteq \omega_1$ that codes $X$ and all reals. 
Using Theorem~\ref{T1.delta} find an iteration $j\colon\mwM\to \fM$ of length $\omega_1$
such that $Y$ is $j(\Wdd(\vec E))$-generic over $\fM$. This forcing has $j(\delta)$-chain condition
and it collapses all cardinals below $j(\delta)$ to $\omega$. 
Therefore $\bbR^{\mwM[Y]}=\bbR$, hence
$\mwM[Y]\models \psi(X)$. 

We now assume $\CH$ and $\mwM\models \varphi^*$  and prove $\varphi$. 
In $\mwM$  fix a condition $p$ in~$\Wdd$ and a name $\dot X$ 
such that $p$ forces  $\psi(\dot X)$. 
By using $\CH$ we can enumerate $\bbR$ as  $r_\xi$, for $\xi<\omega_1$.  
In this proof we shall write $M_0$ instead of~$\mwM$. We shall now describe 
an iteration strategy for player~I. Along with the iteration, player~I constructs generic
filters. 
After player I's strategy is described, we shall run it against \playerII's winning strategy 
for the iteration game and argue that the run of the game produces 
 a well-founded iteration $j\colon M_0\to M^*$ and  
 $G\subseteq j(\cWdd(\vec E))$ generic 
 over $M^*$ such that $p=j(p)$ is in $G$ and that $M^*[G]$ contains all reals. 
All extenders $E$ used in player~I's strategy defined below 
will satisfy $\crit(E)\geq \delta$,  therefore
assuring $j(p)=p$.

Since $M_0\cap \V_{\delta+1}$ is countable,  
in $V$ we can find a  $G_0\subseteq \Wdd(\vec E)$ generic over~$M_0$ and containing $p$. 
Now use the extender in $M_0$ with critical point $\delta$ to find $j_0\colon M_0 \to M_1$. 
By Lemma~\ref{L.reg},  $G_0$ can be extended to a generic filter for $j_0(\cWdd(\vec E))$.

We want to find a generic $G_1\subseteq j_0(\Wdd(\vec E))$ 
extending $G_0$ and such that $r_0\in M_1[G_1]$. 
Let $\delta_0$ be the least Woodin cardinal in~$M_1$ greater than $\delta$ whose
Woodinness is witnessed by (an initial segment of) $j_0(\vec E)$.  
Let $\vec E_1$ consist of generators in $j_0(\vec E)$ witnessing Woodinness of $\delta_0$ 
whose critical points exceed $\delta$. 
Now \playerI\ attempts to find an iteration $i_0\colon M_1\to  M_1^*$ 
using the extenders in $\vec E_1$ such that 
$r_0$ is generic over $i_0(\cW_ {\delta_0\omega} (\vec E_1))$. 
(Recall that  \playerII\  continues playing his winning strategy 
for the iteration game corresponding to $\vec E$ from   $\mwM$, and therefore 
 by Theorem~\ref{T1}  after countably many stages we will assure that $r_0$ is generic.)
 Assume \playerI\ has succeeded in finding $i_0$. Since the critical point never drops below $\delta$, this
 is an iteration of $M_1$ resulting in some $M_1^*$. 
 Since $p$ forces that $(i_0\circ j_0)(\Wdd(\vec E))$ collapses $2^{\delta_0}$ to $\aleph_0$, 
 $i_0(\cWdd{\omega} (\vec E_1))$ can be embedded as a 
 regular subalgebra of the former below $p$. 
 We can therefore use Lemma~\ref{L.Forcing.1} to find  a generic $G_1\subseteq
 (i_0\circ j_0)(\Wdd(\vec E))$
including $G_0$  such that $r_0\in (i_0\circ j_0)(M_1)[G_1]$. 

We proceed in this manner. At the $\alpha$th stage we have an iteration $j^0_\alpha\colon M_0\to M_\alpha$
and $G_\alpha$ is a generic filter for $j_\alpha^0 (\Wdd(\vec E)))$.
\PlayerI\ uses the extender in $M_\alpha$ with critical point $j^0_\alpha(\delta)$ to find 
$j^1_\alpha\colon M_\alpha\to M^1_\alpha$. Let $\delta_\alpha$ be the least Woodin cardinal in $M^1_\alpha$
above $j^0_\alpha(\delta)$. Choose a system of extenders 
$\vec E_{\alpha+1}\subseteq (j^1_\alpha\circ j^0_\alpha)(\vec E)$ 
which witnesses $\delta_\alpha$ is Woodin and such that critical points of all extenders in $\vec E_{\alpha+1}$
exceed $j^0_\alpha(\delta)$. 
Using the genericity iteration theorem, \playerI\ finds an iteration $j^2_\alpha$ of $M^1_\alpha$
such that $r_\alpha$ is generic over $j^2_\alpha(\vec E_{\alpha+1})$. Using Lemma~\ref{L.reg} to absorb
this forcing in quotient, 
Find an iteration $j_\alpha^3\colon M^1_\alpha\to M_\alpha^*$ with critical point
 $j_\alpha^0(\delta)$ such that the filter $G_\alpha$ can be extended to a generic ultrafilter $G_{\alpha+1}$ included in $j_\alpha(\cWdd(\vec E))$, 
  with  $j_\alpha=j_\alpha^3\circ j^1_\alpha\circ  j_\alpha^0$.

Fix the least  Woodin cardinal 
in $M_\alpha^*$ above~$j_\alpha^0(\delta)$. 
Using the algebra with $\omega$ generators on this cardinal  
 and Theorem~\ref{T1} 
  find an iteration of $i_\alpha\colon M_\alpha^*\to M_\alpha^{**}$
that makes  $r_\alpha$ generic for an algebra that is a regular subalgebra 
of $(i_\alpha\circ j_\alpha)(\Wdd(\vec E))$.
Again \playerI\ plays only the extenders $E$ with $\crit(E)\geq j_\alpha^0(\delta)$ so the responses
of \playerII\ result in an iteration $i_\alpha\colon M_\alpha^*\to M_\alpha^{**}$, as required. 
By Lemma~\ref{L.reg}, we can  find
 a generic filter $G_{\alpha+1}\subseteq 
(i_\alpha\circ j_\alpha)(\Wdd(\vec E))$ extending $G_\alpha$ 
  such that $r_\alpha$ belongs 
to~$M_\alpha^{**}[G_{\alpha+1}]$. 

At a limit stage of the construction \playerII\ chooses a maximal branch of the iteration 
tree constructed so far. By the $\delta$-cc, every maximal antichain 
of the image of $\Wdd(\vec E)$ in this model belongs to some
earlier model. 
Therefore the direct limit of the $G_\xi$ corresponding to the 
models on the branch is generic. 
 
This describes a  game which  produces an iteration $j\colon M_0\to M^*$ such that 
$j(\delta)=\aleph_1$ and a $G\subseteq j(\Wdd(\vec E))$ generic over $M^*$
and containing $j(p)=p$. 
By the choice of $p$ and elementarity  $\psi(X)$ holds in $M^*[G]$. Since 
the model~$M^*[G]$ also contains all  reals, $\psi(X)$ holds in~$V$. 
\end{proof}

\begin{cor}[Steel, Woodin]
Assume there exists a fully iterable model $\mwM$
with a countable ordinal $\delta$ that is a measurable  Woodin cardinal in $\mwM$.  
Then every $\Sigma^2_1$ statement true in some forcing extension of~$V$ is
true in every forcing extension of~$V$ that satisfies $\CH$. \qed
\end{cor}

\subsection{A fairly complicated set of natural numbers}\label{S.many-one}\label{S.Oinfty}
The following was pointed out by John Steel. 
Recall that if $\kappa$ is a limit of Woodin cardinals then a set is 
$\kappa$-universally Baire 
if and only if it is $<\kappa$-homogeneously Suslin (\eg, \cite[Theorem~3.3.13]{Lar:Stationary}). 
By $\Hom_\infty$ we denote the pointclass of all homogeneously Suslin sets. 
A sentence $\psi$ is $\bfSigma^2_1(\Hom_\infty)$
if there is $A\in \Hom_\infty$ and formula $\phi(X,Y)$, where $X$ and $Y$ are
second-order variables, such that $\psi$ is of the following form: 
\[
(\exists X\in \Hom_\infty)(H(\aleph_1), \in, A, X)\models \phi. 
\]
By a result of Woodin (\cite[Theorem~5.1]{St:Derived}), if there are class many Woodin cardinals
then every $\bfSigma^2_1(\Hom_\infty)$ sentence is absolute between forcing extensions. 
We shall need only the lightface version of this result. 

Let $\cO_\infty$ denote the set of all $\Sigma^2_1(\Hom_\infty)$ truths, with no real or $\Hom_\infty$
parameters. By identifying $\phi$ with its G\"odel number $\godel\phi$ we consider $\cO_\infty$
as a set of natural numbers. By \cite[Theorem~5.1]{St:Derived}, if there exist class many Woodin 
cardinals $\cO_\infty$ is invariant under forcing.

\begin{cor}Assume there exist class many Woodin cardinals. 
 Let~$\Gamma$ be the set of all  $\Sigma^2_1$ sentences $\phi$ that 
hold in forcing extensions that satisfy~$\CH$. 
If there exists a mouse with a measurable Woodin cardinal and $\Hom_\infty$ 
iteration strategy, then $\Gamma$ is many-one reducible to $\cO_\infty$. 
\end{cor} 

\begin{proof} By Theorem~\ref{T.0}, 
truth of a sentence $\phi$ in $\Gamma$ is equivalent to the 
existence of an iterable mouse with a $\Hom_\infty$ iteration strategy
such that ($\phi$ holds in a forcing extension that satisfies $\CH)^M$. 
\end{proof} 

In the presence of class many Woodin cardinals, 
every generic absoluteness result proved using genericity iterations 
along the lines of Theorem~\ref{T.0} or  Theorem~\ref{T.00}
implies many-one reducibility of the relevant set $\Gamma$ to $\cO_\infty$,  
provided there are class many Woodin cardinals and  the  
 relevant mouse has a $\Hom_\infty$ iteration strategy. 
 Proofs of $\Sigma^2_1$ absoluteness  using stationary tower (\cite{Lar:Stationary}) 
 or Levy collapse followed by forcing with a saturated ideal (\cite{Fa:A_proof}) 
 do not seem to produce many-one reduction of the relevant set of sentences
  to $\cO_\infty$. 

Similarly, determinacy proofs produce many-one reductions of game-quantifier
truths to $\cO_\infty$.  For example,   Neeman's \cite{Nee:Determinacy} 
produces $\Hom_\infty$ strategies for open $\omega_1$-games whose payoff set is 
$\bfPi^1_1$ in the codes 
by transforming a $\Hom_\infty$ iteration strategy for the mouse. 
It therefore shows that $\gopen$ truth 
is many-one
reducible to $\cO_\infty$.  

Recall the definition of a proof in Woodin's $\Omega$-logic 
(see \cite[\S 10]{Wo:Pmax2}, \cite{Woo:Omega}, \cite{BaCaLa},  
or \cite{Lar:Three}). One first assumes there are 
class many Woodin cardinals, and therefore a set of reals is universally Baire
if and only if it is homogeneously Suslin. 
If a set of reals $A$ is universally Baire and $M$ is a transitive model of a large enough fragment 
of ZFC then we say $M$ is \markdef{$A$-closed} if for every $\bbP\in M$ and 
every $\bbP$-name $\tau$ for a real such that $\tau\in M$ the 
set 
$\set{p\in \bbP}{ p\forces \tau\in A}$ belongs to $M$. See \cite[Definition~10.141]{Wo:Pmax2}
or \cite[\S 2.2]{BaCaLa} for more details. For production of $A$-closed models see 
Lemma~\ref{L.uB} and Lemma~\ref{L.Omega}. 
An \markdef{$\Omega$-proof} of $\phi$ 
is a universally Baire set~$A$ such that for every 
$A$-closed model $M$ 
of a large enough fragment of ZFC 
we have $M\models \phi$. If $\phi$ has an $\Omega$-proof  then we write $\vdash_\Omega\phi$.  
Consider the set of all theorems of $\Omega$-logic, 
\[
\cO^\Omega=\set{\phi}{ \vdash_\Omega\phi}. 
\]
This  $\Sigma^2_1(\Hom_\infty)$ set is easily seen to be many-one 
equivalent to~$\cO_\infty$. 

Woodin's $\Omega$-conjecture (see \cite{Woo:Omega}, \cite{Lar:Three})
asserts that $\cO^\Omega$ is the set of G\"odel numbers of statements
true in all forcing extensions. 
Therefore $\Omega$-conjecture implies, and is  essentially equivalent to, the assertion 
 that all generic absoluteness comes via many-one reductions 
to $\cO_\infty$. 

\section{Divergent models of $\AD^+$} 
\label{S.Div}

In the present section we prove another previously unpublished
theorem due to  (surprise, surprise) Woodin. 
Recent sweeping results of Grigor Sargsyan added to the importance of this theorem
by putting it in a chain of implications leading to dramatic lowering of 
the consistency strength of the axiom `$\ADR\axplus\Theta$ is regular'
(see \cite{Sar:Short} or \cite[\S 1.6]{Wo:Pmax2} for more details). 
I would like to thank Grigor for suggesting 
that I include this proof and for clarifying a number of details, in particular 
the ones sketched in \S\ref{S.WlimW}. 
 I am indebted to Hugh Woodin for sketching the proof during a train ride 
from Oberwolfach in January 2011. Koellner's note  
 \cite{Koe:Incompatible} was quite helpful in reconstructing this proof. 
 
Two inner models $N_1$ and $N_2$ of $\AD^+$ are said to be \markdef{incompatible} if they have the same
reals but $\cP(\bbR)^{N_1}\not\subseteq \cP(\bbR)^{N_2}$ and $\cP(\bbR)^{N_2}\not\subseteq \cP(\bbR)^{N_1}$. 
 The terminology is justified by the fact that no inner model that includes both $N_1$ and $N_2$ can satisfy Wadge determinacy.

\begin{thm} \label{T.Div} Assume there exists Woodin limit of Woodin cardinals. 
Then in some forcing extension of $V$ there are incompatible $\AD^+$ models. 
\end{thm}

The remainder of this section contains the proof of Theorem~\ref{T.Div}. 

\subsection{Universally Baire sets and absoluteness}
\label{S.Cohen}
A set of reals $A$ is \markdef{$\kappa$-universally Baire} if there are
trees $S$ and $T$ on $\omega\times \lambda$ for some cardinal $\lambda$ 
such that $S$ projects to $A$, $T$ projects to $\bbR\setminus A$, and 
in every forcing extension by $\bbP\in \V_\kappa$ trees $S$ and $T$ project to complementary 
sets of reals. Set $A$ is \markdef{universally Baire} if it is $\kappa$-universally Baire for all $\kappa$. 
In forcing extension we use these trees as a code for $A$ and identify set $A$ with the
projection of $T$.  We therefore define $A^{V[G]}$ to be $p[T]^{V[G]}$ 
for a generic filter $G\subseteq \bbP$. In \cite{Wo:Pmax2} set $A^{V[G]}$ is denoted $A_G$. 
Good sources for universally Baire sets is \cite{Lar:Stationary} and 
\cite[\S10]{Wo:Pmax2}. 

Given a $\kappa$-universally Baire set $A$ 
we shall need a sufficiently closed with respect to $A$  countable transitive model 
of a large enough fragment of ZFC. 

\begin{lem} \label{L.uB} 
Assume $A$ is $\kappa$-universally Baire and  $\theta>|\V_\kappa|$.  
Expand the language of ZFC by adding a predicate for $A$. 
Then for club many $M\prec H(\theta)$  the transitive collapse $\bar M$ of 
$M$ satisfies the following. 
\begin{enumerate}
\item \label{I.uB.1} $\bar M$ is \markdef{$A$-correct}: $A^{\bar M}=\bbR^{\bar M}\cap A$,
 and
\item \label{I.uB.2} $\bar M$ is \markdef{$A$-absolute}:  
If $\bbP$ is a forcing notion in $(\V_{\bar \kappa})^{\bar M}$ and $G\subseteq \bbP$ 
is $\bar M$-generic, then $A^{\bar M[G]}=\bbR^{\bar M[G]}\cap A$.
\pushcounter
\end{enumerate}
\end{lem} 

\begin{proof} 
Let  $M$  be an elementary submodel of a large enough $H(\theta)$ containing trees $S$ and $T$
that witness $\kappa$-universal Baireness of $A$  and let $\bar M$ be its transitive closure. 
We also denote images of $S,T$, and $\kappa$ under the transitive collapse 
by $\bar S$, $\bar T$, and $\bar\kappa$.  Let~$A^{\bar M}$ denote the projection of $\bar S$ in $\bar M$.  
In $\bar M$ interpret $A$ to be the projection of $\bar T$. 
Since the reals are not moved by the collapsing map, the 
 elementarily of $M$ implies \eqref{I.uB.1} and \eqref{I.uB.2}. 
 \end{proof}  
 
Model $\bar M$ in the conclusion of Lemma~\ref{L.uB} is \markdef{$A$-closed}  (see \S\ref{S.Oinfty}). 

If $A$ is a set of reals and there is a nontrivial elementary embedding of $\LL(A,\bbR)$ into itself
then $(A,\bbR)^\#$ is the theory of the first $\omega$ many indiscernibles for $\LL(A,\bbR)$
with parameters from $\bbR$ and predicate for $A$.  It is naturally identified with a set of reals. 

Theorem~\ref{T.Absolute} is due to Woodin. Its variant for projective sets 
was proved in \cite{Wo:Consistency}. A very similar result also appears in \cite[Theorem~1.13]{woodin83some}
(also compare \cite[Theorem~1.14]{woodin83some} with Corollary~\ref{C.Cohen}  below)
and a variant of part (2) for $\LR$ and proper forcing was proved in \cite{NeZa:Proper}. 

Let us remind the reader that, while large cardinals imply that for a pointclass $\Gamma$ all sets of reals in $\LL(\Gamma,\bbR)$ are universally Baire, these sets are not universally Baire in $\LL(\Gamma,\bbR)$.  A forcing notion $\bbP$ is \markdef{weakly proper} if every countable set of ordinals
in the extension is included in a ground-model countable set. In particular every proper forcing is weakly proper. 

\begin{thm} \label{T.Absolute}
Suppose $A$ is a set of reals such that all sets in $\LL(A,\bbR)$ are $\kappa$-universally Baire 
and $(A,\bbR)^\#$ exists. 
\begin{enumerate} 
\item Then for every forcing notion $\bbP\in \V_\kappa$ and $V$-generic 
flter $G\subseteq \bbP$ there is a generic elementary 
embedding  $j_G\colon \LL(A,\bbR)^{\V}\to \LL(A,\bbR)^{\V[G]}$. 
\item If in addition $\AD$ and $\DC$ hold in $\LL(A,\bbR)$ and $\bbP$ is weakly proper
 then the canonical $j_G$ fixes all ordinals. 
\end{enumerate}
\end{thm}

\begin{proof} (1) This is essentially \cite[Theorem~2.30]{Wo:Pmax2}, where
analogous statement was proved under the assumption that all sets in $\LL(A,\bbR)$
were weakly homogeneously Suslin. In this case the assumption that $(A,\bbR)^\#$ exists
is automatic. The present theorem is proved using 
appropriately modified versions of \cite[Lemma 2.27--2.29]{Wo:Pmax2}, but we shall sketch a proof for the convenience of a reader.  
If $B\subseteq \bbR$ is in $\LL(A,\bbR)$ then we have trees $T$ and $S$ such that $p[T]=B$
and $p[S]=\bbR\setminus B$ and $p[T]$ and $p[S]$ cover $\bbR$ in forcing extension by $\bbP$. 
We therefore let $j_G(B)=p[T]^{\V[G]}$. Note that 
 $(A,\bbR)^\#$ is a countable union of sets in $\LL(A,\bbR)$ (the $n$th set codes the theory of the first
 $n$ indiscernibles), and is therefore universally Baire. 
 If $U$ and $V$ are trees witnessing universal Baireness of $(A,\bbR)^\#$ then 
 an application of Lemma~\ref{L.uB} easily shows that $p[U]^{\V[G]}$ 
 still codes the first order diagram of 
 \[
 (H(\aleph_1)^{\V[G]},\in, j_G(C))
 \]
  where $C$ ranges over universally Baire sets in $\LL(A,\bbR)$. 
  Therefore map $j_G\colon (H(\aleph_1), \in, A)^{\V}\to (H(\aleph_1),\in,A)^{\V[G]}$ 
  is elementary and so is its canonical extension $\bar j_G\colon \LL(A,\bbR)^{\V}\to \LL(A,\bbR)^{\V[G]}$.

(2) Let $j_G$ be as defined in the proof of (1). 
The fact that $j_G$ does not move ordinals is 
 essentially \cite[Theorem~10.63]{Wo:Pmax2}. Proof that $j_G$ does not 
move ordinals relies on a result of Steel (\cite[Theorem~3.40]{Wo:Pmax2}). 
 \end{proof}

An another way to prove variants of Theorem~\ref{T.Absolute} (1) in case when the generic 
filter of $\bbP$ is determined by a single real $x_G$ is to construe 
$\LL(A,\bbR)^{\V[G]}$ as a generic ultrapower of $\LL(A,\bbR)^{\V}$ associated with 
\[
\cU(x_G)=\set{B\in (\cP(\bbR)\cap \LL(A,\bbR))^{\V}}{ x_G \in B}. 
\]
By using universal Baireness of sets in $\LL(A,\bbR)$, $\cU(x_G)$ is a generic  ultrafilter in 
$(\cP(\bbR)\cap \LL(A,\bbR))^{\V}$. In general case one can construe $\LL(A,\bbR)^{\V[G]}$ as
a direct limit of ultrapowers of this sort (see \cite[p. 752]{Wo:Pmax2}). 

By a result of Woodin (\cite{Wo:Pmax2}) there exists a semiproper forcing $\bbP$ such that 
the generic embedding as in Theorem~\ref{T.Absolute} necessarily moves ordinals.

\begin{cor} \label{C.Cohen} 
Assume $A$ satisfies the assumption of Theorem~\ref{T.Absolute} (2) 
for an uncountable cardinal $\kappa$. 
In addition assume $\bbP$ is a forcing notion of cardinality~$\aleph_1$ such 
that every real added by $\bbP$ is added by a countable regular subordering.  
 Then the identity map is an elementary 
embedding of $\LL(\Gamma,\bbR)^{\V}$ into an inner model $N$ that contains all reals. 
\end{cor}

\begin{proof} If $\kappa>\aleph_1$ and $\bbP$ is weakly proper
then this is an immediate consequence of 
Theorem~\ref{T.Absolute}, but it is exactly the case when $\kappa=\aleph_1$ that we shall need. 
By recursion we can extract a regular subordering $\bbQ$ of $\bbP$ 
that is an increasing $\omega_1$-chain of countable regular suborderings 
and such that the quotient $\bbP/\bbQ$ does not add any new reals.
We can write $\bbQ$ as a direct limit of regular suborderings $\bbQ_\xi$, for $\xi<\omega_1$, 
each of which is forcing-equivalent to $\bbC$.

By iterating  Theorem~\ref{T.Absolute} 
 $\omega_1$ times we can conclude that 
$\LL(A,\bbR)^{\V}$ is an elementary 
submodel of $N=\LL(A,\bbR)^{V^{\bbQ}}$. After forcing by quotient $\bbP/\bbQ$ 
no reals are added and therefore $\LL(A,\bbR)^{\V^\bbP}=\LL(A,\bbR)^{\V^{\bbQ}}$. 
This model contains all reals of the extension and therefore satisfies the conclusion of the corollary. 
\end{proof} 


In the following lemma we assume $\bbP$ is a forcing notion with the following 
properties: 
\begin{enumerate}
\item $|\bbP|=\delta$, 
\item $\bbP$ collapses $\delta$ to $\aleph_1$, 
\item \label{I.P.3} 
for every $\bbP$-name $\dot x$ for a real there exists a regular subordering $\bbQ$ of $ \bbP$ such that 
$\dot x$ is a $\bbQ$-name and $|\bbQ|<\delta$, 
\item $\bbP\times \bbP$ has $\delta$-chain condition. 
\pushcounter
\end{enumerate}

\begin{lem} \label{L.product} 
Assume $\bbP$ is as above and $\dot B$ is a $\bbP$-name for a set of reals such that 
$\bbP$ forces $\AD\axplus\DC$ hold in $\LL(\dot B, \bbR)^{V^{\bbP}}$, all sets in $\LL(\dot B,\bbR)$ 
are $\delta$-universally Baire, and $(\dot B,\bbR)^{\#}$  exists. 
Then $\bbP\times \bbP$ forces that there is an elementary embedding which 
fixes all reals and all ordinals from $\LL(\dot B, \bbR)^{V^{\bbP}}$ into 
an inner model of~$V^{\bbP\times\bbP}$ that contains all reals. 
\end{lem}

\begin{proof} Since every real added by $\bbP$ is added by a small forcing
and since~$\bbP$ collapses its own cardinality to $\aleph_1$, this is an immediate 
consequence of Corollary~\ref{C.Cohen}. 
\end{proof} 

\subsection{An iterable model with Woodin limit of Woodin cardinals} \label{S.WlimW} 
As the final preparation for  the proof of Theorem~\ref{T.Div} we outline some properties of 
Neeman's mouse. 
Assume $W=\LL[\vec E]$ is an inner  model with cardinal~$\delta$ which is 
a Woodin limit of Woodin cardinals
as constructed in \cite{Nee:Inner}. Then~$W$ has the following properties. 
\begin{enumerate}
\popcounter
\item\label{I.SS.Div.1} 
 $W$ is an extender model, $W=\LL[\vec E]$, where $\vec E$ witnesses Woodinness of both $\delta$
and cofinally many cardinals below $\delta$. 
\item \label{I.SS.Div.2} For every $\delta$-universally Baire set of reals $A$ in $W$  the model $\LL(A, \bbR)$ satisfies $\AD^+$. 
\item \label{I.SS.Div.4} The transitive collapse of every countable $X\prec (\V_{\delta+1})^W$ is fully iterable with 
 a $\delta$-universally Baire iteration strategy. 
 \pushcounter
 \end{enumerate}
 Clause \eqref{I.SS.Div.1} is a result of  \cite{mitchell94iteration}. Clause \eqref{I.SS.Div.2} is a result 
 of Neeman \cite{neeman02optimal}. 
Clause~\eqref{I.SS.Div.4} can be proved by using methods from  
  \cite{Ste:DerivedModels} where similar statements were proved. 

We shall need conditions \eqref{I.SS.Div.3} and \eqref{I.SS.Div.5} given below, while  
\eqref{I.SS.Div.6} and \eqref{I.SS.Div.7} will not be used outside of the present subsection. 
Let us call a countable model of the form $\LL[\vec E]$ that contains
a Woodin limit of Woodins \markdef{premouse}.  
Every premouse $M$ has the canonical 
relative constructibility well-ordering~$<_M$ of the reals. 
 We shall refer to an 
iterable premouse as a \markdef{mouse}.

 Recall that $\dGuB$ denotes the class of all $\delta$-universally Baire sets of reals. 
A sentence $\psi$ is $\bfSigma^2_1(\dGuB)$
if there is $A\in \dGuB$ and formula $\phi(X,Y)$, where $X$ and $Y$ are
second-order variables, such that $\psi$ is of the following form: 
\[
(\exists X\in \dGuB)(H(\aleph_1), \in, A, X)\models \phi. 
\]
Under our assumption that $\delta$ is a limit of Woodin cardinals $<\delta$-homogene\-ously Suslin 
is equivalent to $\delta$-universally Baire and this pointclass 
coincides with $\Sigma^2_1(\Hom_{<\delta})$ 
considered in \S\ref{S.many-one}. 

\begin{enumerate}
\popcounter
\item \label{I.SS.Div.6} If $M$ and $N$ are $(\omega_1+1)$-iterable mice then either 
$\bbR^M\subseteq \bbR^N$ and~$<_M$ is an initial segment of $<_N$ 
or $\bbR^N\subseteq \bbR^M$ 
and  $<_N$ is an initial segment of~$<_M$.  
\item \label{I.SS.Div.7} The set of $(\omega_1+1)$-iterable mice is $\Sigma^2_1(\dGuB)$. 
 \pushcounter
 \end{enumerate}
Given $M$ and $N$ as in \eqref{I.SS.Div.6}, a standard comparison argument (that can be reconstructed
from the proof of Theorem~\ref{T1}\footnote{Actually the proof of Theorem~\ref{T1} was based on the proof of Comparison Lemma}) shows that $M$ and $N$ have  countable iterations
$j\colon M\to M^*$ and $k\colon N\to N^*$ such that $M^*$ is an initial segment of $N^*$ or 
$N^*$ is an initial segment of $M^*$. In either case, the well-orderings $<_{M^*}$ and $<_{N^*}$ 
coincide. Since the reals of a mouse and their well-ordering are unaffected by iteration, 
\eqref{I.SS.Div.6} follows. 

Now we prove \eqref{I.SS.Div.7}. A set $\Sigma$ is an $\omega_1$-iteration 
strategy if and only if for all countable iteration trees $T$ obtained when \playerII\ obeys $\Sigma$
all branches chosen by $\Sigma$ result in well-founded models. Therefore the assertion 
that~$\Sigma$ is an $\omega_1$-iteration strategy is projective in $\Sigma$. 
Therefore stating that a countable transitive model $M$ is 
of the form $\LL[\vec E]$ and is $\omega_1$-iterable 
with a $\delta$-universally Baire iteration strategy is $\Sigma^2_1(\dGuB)$.

By Lemma~\ref{L.uB.strategy}, every such model is $\omega_1+1$-iterable, and this completes 
the proof of \eqref{I.SS.Div.7}. 
 
 \begin{enumerate}
 \popcounter
\item \label{I.SS.Div.3} There is a $\Sigma^2_1(\dGuB)$ good well-ordering of the reals in $W$. 
 \item\label{I.SS.Div.5} 
  For every real $r$ there is a $\delta$-universally Baire set $A$ such that $r$ is ordinal definable in $\LL(A,\bbR)$. 
 \pushcounter
 \end{enumerate}
 By \eqref{I.SS.Div.4}  every real belongs to an $\omega_1+1$-iterable mouse with a 
 $\delta$-universally Baire strategy. Since canonical well-orderings of the reals in 
 such mice cohere by \eqref{I.SS.Div.6}, we can define well-ordering 
 by letting $x<y$ if there exists an $\omega_1+1$-iterable mouse $M$ with a $\delta$-universally Baire iteration strategy such that $x<_M y$. By \eqref{I.SS.Div.7} 
 this is a $\Sigma^2_1(\dGuB)$ statement.

  We finally prove \eqref{I.SS.Div.5}. 
  Let   $M$ be an $\omega_1+1$-iterable mouse containing real~$x$ and let 
  $\Sigma$ be its $\delta$-universally Baire  strategy. Let $\alpha$ be such that~$x$ is 
  the $\alpha$-th real in $<_M$. By \eqref{I.SS.Div.6}, $x$ is the $\alpha$-th 
  real in every mouse $N$ such that $x\in N$ and $\Sigma$ is 
  an $\omega_1+1$-iteration   strategy for $N$.  
  We want to show that  $x$ is ordinal-definable 
  from $\alpha$ in $\LL(\Sigma\rs\omega_1, \bbR)$. Since $\Sigma\rs\omega_1$ need not be $\delta$-universally Baire in $\LL(\Sigma\rs\omega_1,\bbR)$, Lemma~\ref{L.uB.strategy} does not apply. 
  Nevertheless, $\Sigma\rs\omega_1$ can be extended to an $\omega_1+1$-iteration strategy 
  as follows. Since $\LL(\Sigma\rs\omega_1, \bbR)$  is a model of $\AD$,   $\aleph_1$ 
  is a measurable cardinal in it. In particular it has the tree property and after an iteration game 
  of length $\omega_1$ \playerII\ can choose an $\omega_1$ branch of the iteration tree. As a direct limit of 
  well-founded models of length $\omega_1$, the model corresponding to this branch is well-founded.

\subsection{Proof of Theorem~\ref{T.Div}} \label{S.S.Divergent} 
Assume that $V$ is Neeman's model described in \S\ref{S.WlimW} in which  $\delta$ is a Woodin
 limit of  Woodin cardinals and 
 continue numbering of formulas started in~\S\ref{S.WlimW}. 
  Let $M$ be the transitive collapse of an elementary submodel $X$ of $H(\lambda)$ for a large enough $\lambda>\kappa$. 
  An initial segment of $M$ is an initial segment of  $X\cap \V_{\kappa+1}$  and it therefore 
  satisfies \eqref{I.SS.Div.4}. This iteration strategy clearly gives an iteration strategy for~$M$.  
 Pick a Cohen real $c$ over $M$. Now fix a $\delta$-universally Baire set $A$ as in~\eqref{I.SS.Div.5} 
 so that $c$ is ordinal definable in $\LL(A,\bbR)$. Since $W$ satisfies  $\CH$, both $A$ and $A^\#$ can 
 be identified with subsets of $\omega_1$.

 Let  $\bbR\rs\alpha$
 be the first $\alpha$ reals in the 
 well-ordering of the reals as in \eqref{I.SS.Div.3} and write $B\rs \alpha=B\cap (\bbR\rs\alpha)$ for 
 $B\subseteq\bbR$. 
 By applying Theorem~\ref{T.WlimitW} to a $C\subseteq \omega_1$ that codes $A,A^\#$ and 
 trees witnessing $\aleph_1$-universal Baireness of all sets in $\LL(A,\bbR)$, we can find an iteration 
 $j\colon M\to M^*$ such that  with $\alpha=j(\delta)$ the following hold: 
 \begin{enumerate}
 \popcounter
 \item\label{I.c.0}   $\LL(C\rs \alpha, \bbR\rs \alpha)\prec \LL(C,\bbR)$, 
 \item \label{II.WlimW.2} the pair 
 $(C\rs\alpha, \bbR\rs\alpha)$ is $j(\cWdd)$-generic over $M^*$, 
 \item  every real in the extension is added by a small forcing, 
 \item \label{I.allreals} the reals of the forcing extension from \eqref{II.WlimW.2} 
 are equal to $\bbR\rs \alpha$.
 \pushcounter 
 \end{enumerate}
 We also assure that $\alpha$ is large enough to have $c\in \bbR\rs\alpha$.  
 Then for generic $G\subseteq j(\cWdd)$ as in \eqref{II.WlimW.2} we have that
 in $M^*[G]$  real $c$ is ordinal definable over a model 
 of the form  $\LL(B,\bbR)$ for some $\aleph_1$-universally Baire
 set $B$, and this model contains all reals of the extension by \eqref{I.allreals}. This model  is a model of $\AD^+$ as witnessed by $A^\#$. 
 Note that, since $M$ and $M^*$ have the same reals,~$c$ is Cohen over~$M^*$. 
 
 By elementarity between $M$ and $M^*$ we can fix a condition $p$ in $(\cWdd)^M$ which forces
 that there exist a real $c$, an $\aleph_1$-universally Baire set $B$, 
 and a countable ordinal~$\alpha$ such that 
 \begin{enumerate}
 \popcounter
 \item \label{I.c.1} $c$ is the $\alpha$-th real in the well-ordering of ordinal definable reals in $\LL(B,\bbR)$, 
\item \label{I.c.1.5} $\LL(B,\bbR)\models \AD^+$, 
 \item  every new real belongs to a forcing extension by a forcing of cardinality $<\delta$, 
 \item  \label{I.c.3} $c$ does not belong to any dense $G_\delta$ set  coded in $M$, 
  \end{enumerate}
Let  $G_1\times G_2$ be a  $\cWdd\times \cWdd$-generic filter over $M$ 
below $(p,p)$. 
We therefore have an $\aleph_1$-universally Baire set $B_1$ and real $c_1$ in $M[G_1]$ such that 
in $\LL(\bbR,B_1)^{M[G_1]}$ conditions \eqref{I.c.1}--\eqref{I.c.3} hold. 
We also  have a universally Baire set~$B_2$ and real $c_2$ in $M[G_2]$ such that 
\eqref{I.c.1}--\eqref{I.c.3} hold with the same~$\alpha$. 

By Hjorth's Lemma~\ref{L4+}  product  $\cWdd\times\cWdd$ has $\delta$-cc and 
$\cWdd$ satisfies the assumptions of Lemma~\ref{L.product}. 
By applying Lemma~\ref{L.product} twice, we find elementary embeddings
\[
j_1\colon \LL(\bbR,B_1)^{M[G_1]}\to \LL(\bbR,B_1')^{M[G_1][G_2]}
\]
and 
\[
j_2\colon \LL(\bbR,B_2)^{M[G_1]}\to \LL(\bbR,B_2')^{M[G_2][G_1]}.
\] 
Both $j_1$ and $j_2$ fix all reals and all ordinals. 

Now assume that in $M[G_1][G_2]$ there are no divergent models of $\AD^+$. 
This applies to images of  $j_1$  and $j_2$. Therefore one of these models includes
the other, and in this model both $c_1$ and $c_2$ are the $\alpha$-th ordinal definable 
real in the well-ordering of ordinal-definable reals. 
We therefore must have that $c_1=c_2$. This   
implies $c_1\in M$, contradicting the fact that $c_1$  
is a Cohen real over $M$.

\section{Other applications} 
\label{S.Sigma22}
A positive answer to the following  (modulo sufficient large cardinals) 
was conjectured by John Steel (see also \cite{Wo:Beyond}). 

\begin{quest} \label{Q.Sigma-2-2}
Assume $\varphi$ is a $\Sigma^2_2$ sentence such that $\CH\axplus\varphi$ 
holds in some forcing extension. Is it true that $\varphi$ holds in every forcing extension 
that satisfies $\lozenge$? 
\end{quest}

Several partial positive results were proved in \cite{DoSch:Extender} 
 and in \cite{FaKeLaMa:Absoluteness}.  
By a result of Woodin (see \cite{KeLaZa:Regular}), if there is a measurable 
Woodin cardinal then there is a forcing $\bbP$  that simultaneously 
forces every $\Sigma^2_2$ sentence~$\varphi$ that 
holds in some forcing extension satisfying $\CH$.
The forcing $\bbP$ is the iteration of the collapse of $2^{\aleph_0}$ to $\aleph_1$ 
and an another forcing notion. 
It is not known whether the collapse of $2^{\aleph_0}$ to $\aleph_1$ 
alone suffices for this conclusion. 

The extender algebra has many other applications but it is time to finish (\compare{} \cite[p. 155]{Mil:Descriptive}). 
For example, in \cite{Lar:Three} a proof that the $\Omega$-conjecture is true in many  
inner models (for example, 
Neeman's model briefly described in \S\ref{S.WlimW}) was presented. One of the key points in the proof 
is Lemma~\ref{L.Omega} below. 
If~$A$ is a universally Baire set of reals and $M$ is a transitive model of a large enough fragment of ZFC, 
we say that $M$ is \markdef{$A$-closed} if for every generic extension M$[G]$ of $M$ we have that 
$A\cap M[G]\in M[G]$. This condition is closely related to the conclusion of Lemma~\ref{L.uB}, 
the only difference being that in the present situation 
 the language of ZFC is not expanded by adding a predicate for $A$. 
 Note that the  assumption of the following lemma 
 is a consequence of clause \eqref{I.SS.Div.4} in  \S\ref{S.WlimW}.

\begin{lem} \label{L.Omega} Assume that $\delta$ is a Woodin cardinal
and $\kappa>\delta$ is an inaccessible cardinal such that the transitive collapse of 
 every elementary submodel $H\prec \V_\kappa$ is fully iterable with a universally Baire iteration strategy. 
Then for every  universally Baire set of reals $A$ there is a fully iterable $A$-closed 
transitive model. \qed
\end{lem} 

The proof of Lemma~\ref{L.Omega}  is identical to the proof of Lemma~\ref{L.uB}. 
Its conclusion,  together with an application of Theorem~\ref{T1},  implies $\Omega$-conjecture 
(see \cite[Theorem~9.2]{Lar:Three}). 

\nocite{cabal83seminar}
\nocite{handbook06setthy}

\bibliographystyle{cabal}
\bibliography{ifmain,cab-strings,cabal,thisvolume,cab-crossrefs}
\bibtexhack

\end{document}